\documentclass[12pt]{amsart}
\usepackage[T1]{fontenc}
\usepackage{tikz-cd}
\usepackage{amsmath}
\usepackage{calligra}
\usepackage{amssymb}
\usepackage{amsthm}
\usepackage{xcolor}
\usepackage{appendix}
\usepackage[margin=1.2in]{geometry}
\usepackage{hyperref}
\usepackage{mathrsfs} 
\usepackage{yfonts}
\usepackage{egothic}
\usepackage{aurical}
\usepackage{enumitem}

\usepackage{pifont}

\newtheorem*{thm-non}{Theorem}
\newtheorem{thm}{Theorem}
\numberwithin{thm}{subsection}
\newtheorem{dfn}{Definition}
\numberwithin{dfn}{subsection}
\newtheorem{prop}{Proposition}
\numberwithin{prop}{subsection}
\newtheorem{cor}{Corollary}
\numberwithin{cor}{subsection}
\newtheorem{lem}{Lemma}
\numberwithin{lem}{subsection}
\newtheorem{rem}{Remark}

\newtheorem{D-prop}{Definition-Proposition}

\newcommand{\con}[1]{{\stackrel{\scriptscriptstyle{#1}}{\nabla}}\phantom{}}

\newcommand{\D}{\mathscr{D}}

\newcommand{\F}{\mathscr{F}}
\newcommand{\ft}{\frak{t}}

\newcommand{\fg}{\frak{g}}

\DeclareMathAlphabet{\mathcalligra}{T1}{calligra}{m}{n}

\newcommand{\C}{\mathbb{C}}
\newcommand{\R}{\mathbb{R}}
\newcommand{\hH}{\mathbb{H}}
\newcommand{\oO}{\mathbb{O}}
\newcommand{\K}{\mathbb{K}}

\newcommand{\cJ}{\mathcal{J}}

\newcommand{\cT}{\mathcal{T}}

\title[LG models, Monge-Ampère domains and (pre-)Frobenius manifolds]
{Landau--Ginzburg models, Monge--Ampère domains and (pre-)Frobenius manifolds}\author{Noemie C. Combe}
\begin{document}

\begin{abstract}

Kontsevich suggested that the Landau–Ginzburg (LG) model presents a good formalism for homological mirror symmetry. In this paper, we propose to investigate the LG theory from the viewpoint of Koopman–von Neumann's  construction. 
New advances are thus provided, namely regarding a conjecture of Kontsevich--Soibelman (on a version of the Strominger--Yau--Zaslow mirror problem). We show that there exists a Monge--Ampère domain $\mathscr{Y}$, generated by a space of density of probabilities, parametrising mirror dual Calabi--Yau manifolds. This provides torus fibrations over $\mathscr{Y}$. The mirror pairs are obtained via the Berglund--Hübsch--Krawitz construction. We also show that Monge--Ampère manifolds are potential pre-Frobenius manifolds. Our method allows to recover certain results concerning Lagrangian torus fibrations. We illustrate our construction on a concrete toy model, which allows us, additionally, to deduce a relation between von Neumann algebras, Monge--Ampère manifolds and pre-Frobenius manifolds.

\end{abstract}
\maketitle
\setcounter{tocdepth}{1}

\medskip 



{\bf 2020 Maths Subject Classification:} Primary: 53A15, 53C35, 35J96, 35Q99,14J33, 53D45 Secondary: 70GXX

\section{Introduction} 
Kontsevich suggested that the Landau–Ginzburg (LG) model presents a good formalism for homological mirror symmetry~\cite{Kont98}. In this paper, we propose to consider the LG theory from the standpoint of the construction of Koopman--von  Neuman (KvN). This allows us to obtain advances regarding the conjecture~\cite{KoS01} on a version of the Strominger--Yau--Zaslow mirror conjecture and to recover certain results on Lagrangian torus fibrations such as depicted in \cite{AAK}. 

The  Kontsevich--Soibelman conjecture asserts that in the limit, both mirror dual manifolds $X$ and $X^\vee$ become fiber bundles with toroidal fibers, over the same base $\mathscr{Y}$. The latter is based on a version of the Strominger--Yau--Zaslow (SYZ) conjecture~\cite{SYZ}, suggesting a certain duality between torus fibrations. 
 
~
 Using LG theory/LG models {\it \`a la} Koopman-von Neumann, we show that there exists a Monge--Ampère domain $\mathscr{Y}$ parametrising mirror dual Calabi--Yau manifolds $X$ and $X^\vee$ producing Lagrangian torus fibrations over $\mathscr{Y}$  (Theorem~\ref{T:FINAL}). The mirror pairs are obtained by the Berglund--Hübsch--Krawitz construction; the Monge--Ampère domain is a given by a space of density of probabilities.

The  difference between Koopman-von Neumann(KvN) version of LG theory and the LG model defined in \cite{AAK,AAEKO,Cecotti,CR} relies in the fact that the KvN--LG theory provides a Hilbert space, which corresponds to the space of states, and this Hilbert space generates a space of probability densities. The KvN--LG theory implies the LG model but the converse is not true. Hence, our construction allows to recover results on LG models related with the  Lagrangian torus fibrations.  As a corollary of our result, we are able to recover results ~\cite[Sec. 2--3]{AAK} in a completely different way. 

{\it $\diamond$ By an abuse of language,  when we mention the LG theory it means the LG theory from Koopman-von Neumann's  viewpoint.}

~

 We cite some related  works of \cite{AAK,AAEKO,KoS01,LW,Man98,SYZ,BH,K}  where LG models, Monge--Ampère manifolds and torus fibrations are used in a Homological Mirror Symmetry perspective.

\subsubsection{}
In this paper, on the one hand, we prove that using LG theory and Koopman-von Neumann's approach, one can show that
the weighted projective space, corresponding to the Hilbert space of states, is parametrised by a real Monge--Ampère manifold, forming a pre-Frobenius manifold. As a result, it follows that a pair of mirror-dual Calabi--Yau manifolds can be parametrised by one Monge--Ampère manifold, forming Lagrangian torus fibrations. 
All our results are proved using methods of {\it affine geometry}.

\smallskip

\subsubsection{} On the other hand, the LG theory provides a state space of an $n$-dimensional quantum system, represented as the set of all $n\times n$ positive semidefinite complex matrices of trace 1 known as density matrices. We consider an enriched version of this object, by allowing the space of all $n\times n$ positive definite matrices to have coefficients in a division algebra (without restrictions on the trace) and show that this space is not only a potential pre-Frobenius domain but also an Elliptic Monge--Ampère domain. This domain contains a Frobenius manifold generated by an algebraic torus. 

The considered model has many applications. For instance by taking the real cone, we provide a Monge--Ampère domain and this space parametrises complex tori, forming the simplest example of Calabi--Yau manifolds. This is reminiscent to the construction \`a la Strominger--Yau--Zaslow (SYZ)  in \cite[Sec. 8.3]{KoS01}. From our results, it follows that the complex cone provides a bridge from von Neumann Algebras to Monge--Ampère manifolds and to Frobenius manifolds.

\subsubsection{}
We mark a terminological difference between the LG {\it theory} and LG {\it model}. In this paper, LG theory refers to the original construction given by Landau and Ginzburg for superconductivity, expressed using the approach of Koopman-von Neuman. The LG model  refers to developments of \cite{Vafa,Cecotti,CR,Chiodo,LW} and many others.

\subsubsection{}
In this article, we adopt the definition of Frobenius manifolds given in  \cite[p.19]{Man99}, where a Frobenius manifold is a potential pre-Frobenius manifold satisfying the associativity condition. Such a framework of Frobenius manifolds inscribes itself as a continuation of the vision started by Yu. Manin in \cite[p.19-20]{Man99} and forms a continuation for the Hessian geometry school \cite{Shi84,SY,Kito,To04}. 

We highlight that investigating relations between sources of Frobenius manifolds is part of the mirror problem \cite[p.3]{Man98}: 
`` {\it Isomorphisms of Frobenius manifolds of different classes remain the most direct expression, although by no means the final one, of various mirror phenomena. From this vantage point, [...]  one looks for isomorphisms between Frobenius manifolds (and their submanifolds) constructed by different methods.}"

Therefore, a statement showing that the LG theory/LG model forms a certain bridge between classes of Frobenius manifolds is fundamental. 

\subsection{Main result}
We demonstrate that there exists a real Monge--Ampère manifold  $\mathscr{Y}$ parametrising a Hilbert space $ \mathfrak{H}$, obtained from the Landau–Ginzburg theory and Koopman--von Neumann theory. This construction $\pi: \mathfrak{H}\to\mathscr{Y}$ generates weighted projective spaces and a torus fibration $\pi: \mathfrak{H}\to \mathscr{Y}$. From well-known works \cite{Vafa,Cecotti}, resulting in the famous Landau--Ginzburg/Calabi--Yau correspondence, one can construct Calabi--Yau manifolds/orbifolds in weighted projective spaces.  
Relying on this, we are able to show that there exist pairs of mirror-dual Calabi--Yau manifolds/orbifolds $(X,X^\vee)$, lying in their respective weighted projective space and parametrized by a real Monge--Ampère manifold  $\mathscr{Y}$, where $2dim_\R(\mathscr{Y})=dim_\R(X)$ such that this construction provides a torus fibration: 
\[
\begin{tikzcd}
&\mathfrak{H} \arrow[dr]\arrow[dl] &\\
X \arrow[dr,"Torus \, fiber"]\arrow[rr,"Mirror\, pair"] & & X^\vee \arrow[ll] \arrow[dl] \\
&\mathscr{Y}  &
\end{tikzcd}
\]

The mirror pairs are constructed via the Berglund--Hübsch--Krawitz~\cite{BH,K} method  and the Theorem 4 of Chiodo--Ruan~\cite{CR}. The Monge-Ampère domain is a space of probability densities. This space also satisfies the axioms of a potential pre-Frobenius domain.

 The sketch of proof is the following. 

\subsection{Sketch of proof}

\begin{enumerate} [label=\roman*., itemsep=0pt, topsep=0pt]
\item In Sec.~\ref{S:Hessian} we introduce Hessian manifolds because it allows to define Frobenius manifolds. Indeed, we show that affine Hessian manifolds satisfy the Associativity Equation if they are flat, see Lemma \ref{L:1.5.1} and Proposition~\ref{P:Codazzi1}.

\item In Sec. \ref{S:2}, we discuss the geometrization of the WDVV equation, playing an important role in the notion of Frobenius manifolds. Through the notion of {\it potential pre-Frobenius manifolds} we discuss an intermediate structure, which is formed from manifolds satisfying the five first axioms of a Frobenius manifold, the last axiom is not necessary (associativity condition). In this sense, a Frobenius manifold is a potential pre-Frobenius manifold, but the converse is not true. We show that a Hessian manifold forms a class of potential pre-Frobenius manifolds.

\item In Sec.\ref{S:PDE} we introduce Monge–Ampère domains/manifolds and in Theorem \ref{T:main} prove that the Elliptic Monge--Ampère manifold forms a {\it potential pre-Frobenius manifolds}. This leads us to a discussion on the existence of isometrically immersed Frobenius (sub)manifolds in an  Elliptic Monge--Ampère manifold proved in Theorem~\ref{T:preF} and Proposition~\ref{P:Codazzi2}. Monge--Ampère manifolds are present in \cite[Sec.3.2]{KoS01}, in the scope of studying the geometric mirror symmetry {\it à la} Strominger--Yau--Zaslow (SYZ).
\item In Sec.~\ref{S:LG1}, we investigate the Landau--Ginzburg theory from the viewpoint  of Koopman-von Neuman\cite{Koo,vN} as it turns out to be useful to prove our statements. The LG theory comes from the BCS superconductors and relativistic extension of abelian Higgs models. They attract the attention of both Condensed Matter and High Energy communities.

On the other side, a Landau-Ginzburg model is summarised as a pair $(X,W)$ consisting of a non-compact K\"ahler manifold $X$ and a holomorphic Morse function $W : M \to \C$, called superpotential. The LG model applies to the cases of Saito spaces, $L^2$-Hodge structures and Calabi--Yau manifolds~\cite{Cecotti,Chiodo,CR,LW,Wil,To04,Vafa}.

Using the original version of the LG theory and Koopman-von Neumann's construction, we show in Theorem \ref{T:LG} that there exists a real Monge--Ampère domain parametrising a weighted projective space (the latter is provided by the LG theory). In particular, there exists a torus bundle $\pi:\mathfrak{H}\to \mathscr{Y}$, where $\mathscr{Y}$ is a real Monge–
Ampère domain and $\mathfrak{H}$ is the Hilbert space coming from the LG theory. The space $\mathfrak{H}$ is formed from square integrable functions with respect to a density function, defined over a phase space. The Monge--Ampère domain is a space of density of probabilities. 

Applying Theorem \ref{T:main} we deduce that the weighted projective space, provided by the LG theory, is parametrised by a  potential pre-Frobenius manifold, see Corollary \ref{T:END}. In particular, Theorem \ref{T:FINAL} implies that there exists a real Monge-Ampère domain $\mathscr{Y}$ parametrising a pair of mirror-dual Calabi--Yau hypersurfaces in their respective weighted projective space. This construction forms a torus fibration. The real dimension of the  Calabi--Yau hypersurfaces is twice the real dimension of $\mathscr{Y}$.  From our construction it follows that we can recover certain results of \cite{AAK}. 

This ends our proof. 

\end{enumerate}

\subsection{LG Toy model}\label{S:EXIN}
In Sec.~\ref{S:ToyModel1}--Sec.~\ref{S:Pre-Frob} we construct a toy model, in order to have a better understanding of the relation between the Landau--Ginzburg theory, Frobenius manifolds and Calabi--Yau manifolds in the framework of the SYZ and KS conjecture.

This toy model is given by the cone $\overline{\mathscr{P}}_n(\K)$ ($\K$ is a real division algebra) of symmetric(/hermitian) semi-positive definite matrices of size $n\times n$ with coefficients in a real division algebra $\K$. 
The real division algebra $\K$ includes real numbers $\R$, complex numbers $\C$, quaternions $\hH$ or octonions $\oO$ (however for the latter $n=3$). For simplicity, we omit the singular locus given by  $x^TAx=0$, where $A$ is a symmetric matrix given thus the open cone $\mathscr{P}_n(\K)$. We show that $\mathscr{P}_n(\K)$  are Monge--Ampère domains. This has several implications. 

\begin{enumerate}
    \item We show that the cone $\mathscr{P}_n(\R)$ is a real Monge--Ampère domain of real dimension $n$ parametrizing principally polarized tori of complex dimension $n$, which are Calabi--Yau manifolds. This furnishes an example of a torus fibaration.
    \item The complex cone $\overline{\mathscr{P}}_n(\C)$ is the space of semi-positive definite hermitian matrices. It contains a hypersurface given by the space of semi-positive definite $n \times n$ matrices of trace 1. This is a very important object in relation with the Landau--Ginzburg theory, as it is used to represent the space of quantum states of an $n$-dimensional quantum system. By showing that $\mathscr{P}_n(\C)$ is a complex Monge--Ampère domain, it implies (by YAU) that the LG model provides Calabi--Yau manifolds. 
  
    \item We prove generally that an irreducible cone $\mathscr{P}_n(\K)$ generates a Monge--Ampère domain (Proposition~\ref{P:pdsmPreF}), and that it forms a pre-Frobenius domain (Theorem \ref{T:preF}, Proposition \ref{P:pre-Fro}). From Theorem~ \ref{T:Hel}, Propositions~\ref{P:exist}--\ref{P:Matrix} it follows that they contain an isometrically immersed Frobenius manifold, being an algebraic torus (Theorem~\ref{T:Frobenius}, Corollary~\ref{T:Final}). The Frobenius manifold is given by $\mathscr{F}=\exp \tilde{\frak{a}}$, where $\tilde{\frak{a}}$ is a Cartan subalgebra (Proposition \ref{P:Matrix}). 

\end{enumerate}

\smallskip 

{$\diamond$ Whenever the context is clear we write $\mathscr{P}$ instead of  $\mathscr{P}_n(\K)$.}

\

In particular, we are interested in the following two cases: the cone $\mathscr{P}(\R)$ over $\R$ and the cone $\mathscr{P}(\C)$ over $\C$.
\begin{itemize}
\item[\ding{169}] In the first case, the cone over $\R$ illustrates the case of a Monge--Ampère domain parametrizing principally polarized tori of dimension $n$, which are Calabi--Yau manifolds. We have two different proofs of the existence of a pre-Frobenius structure on $\mathscr{P}_n(\R)$. One method relies on the fact that $\mathscr{P}_n(\R)$ forms an (elliptic) Monge--Ampère domain; the other via symmetric spaces ($\mathscr{P}_n(\R)$ can be identified with the quotient of Lie groups: $Gl_n(\R)/O_n(\R)$).


\item[\ding{169}]  In the second case (the cone over $\C$) applying the main Theorem of \cite{Connes} allows to deduce the existence of a relation connecting pre-Frobenius domains, Monge--Ampère domains and von Neumann algebras.

\end{itemize}

\smallskip 

\thanks{
{\bf Acknowledgements} This research is part of the project No. 2022/47/P/ST1/01177 co-founded by the National Science Centre and the European Union's Horizon 2020 research and innovation program, under the Marie Sklodowska Curie grant agreement No. 945339 \includegraphics[width=1cm, height=0.5cm]{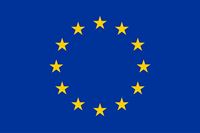}.
I am very much obliged to IHES for supporting my stay in 2022 (November--December), where stimulating conversations with Maxim Kontsevich allowed me to the create this article. I am thus very grateful to Maxim Kontsevich. I
would like to thank Andrzej Dabrowski for proof reading the sections 5 and 6 of this article. My gratitude goes to Hanna Nencka for introducing me to the superconductivity Landau Ginzburg theory à la Koopman--von Neuman. 
}

\section{Hessian manifolds, Codazzi tensors and Monge--Ampère manifolds}\label{S:Hessian}

\subsection{Affine and Hessian manifolds}\label{S:1.1}
\subsubsection{}\label{S:2.2.1}
Let $M$ be a smooth $n$-dimensional manifold. An {\it affine structure} on $M$ is defined by a collection of coordinate charts $\{U_a,\phi_a\}$, where $\{U_a\}$ is an open cover of $M$ and $\phi_a:U_a\to \R^n$  is a local coordinate system, such that the coordinate change $\phi_b\circ \phi_a^{-1}$ is an affine transformation of $\phi_a(U_a\cap U_b)$ onto $\phi_b(U_a\cap U_b)$. Manifolds with an affine structure are usually called affine manifolds or affinely flat manifolds. We highlight the following fact. 
An affine structure on $M$ induces a flat, torsionless affine connection $\con{0}$ on $M$, and reciprocally.

\begin{itemize}
\item[\ding{169}] {\it Throughout the paper: $\con{0}$ refers to  a flat, torsionless affine connection on a manifold. }
\end{itemize}

\subsubsection{Affine vector fields}
We recall the definition and properties of affine vector fields, since affine structures have already been discussed in Sec.\ref{S:2.2.1}.  In a  Frobenius manifold those affine vector fields coincide with Euler vector fields. Throughout our work the manifolds are assumed to be finite-dimensional.

In the context of affine spaces and affine manifolds, affine vector fields  appear naturally. We recall their properties. From those properties, it follows that affine vector fields coincide with the Euler vector fields in the context of Frobenius manifolds and $F$-manifolds.

Let $\mathbb{E}$ be an affine space. Let us denote the Lie algebra of vector fields on $\mathbb{E}$ by $\Upsilon(\mathbb{E})$. A vector field $E$ on $\mathbb{E}$ is {\it affine} if it generates a one-parameter group of affine transformations. The notation $E$ of the vector field refers to the {\it Euler vector field.} In Sec.~\ref{S:Euler} we will see that on a Frobenius manifold Euler vector fields are affine vector fields.

\smallskip 

Recall the following classical result. 
\begin{prop}\label{P:Euler}
The following statements are equivalent.
\begin{enumerate}
 \item $E$ is an {\it affine} vector field.
 \item $\con{0}(\con{0} E)=0$.
    \item $\con{0}_Y\con{0}_Z E=\con{0}_{\con{0}_YZ}E$ for all vector fields $Y,Z\in \Upsilon(M)$.
    \item The coefficients of $E$ are affine functions. Given, $$E=\sum_m E^{m}\partial_m$$ we have $E^{m}=a^m_j x^j + b^m$, where $a^m_j$ and $b^m$ are constants in $\R$.
\end{enumerate}
\end{prop}

Let $M$ be an affine manifold. Denote the Lie algebra of vector fields on $M$ by $\Upsilon(M)$. 
Let $\texttt{aff}(M)$ be the space of affine endomorphisms of $M$. Then, this forms a sub Lie algebra of the Lie algebra $\Upsilon(M)$. A vector field $E$ on $M$ is {\it affine} if in local coordinates it appears as a vector field in $\texttt{aff}(M)$. 

Given $\con{0}$ an affine flat torsionless connection, corresponding to the affine structure on $M$, the vector field $E\in \Upsilon(M)$ is {\it affine}  if and only if $$\con{0}_E(X)=[E,X]=Lie_E(X)$$ which is the Lie derivative of $X$.

\subsubsection{Hessian manifolds}
A Hessian manifold $(M,g,\con{0})$ is a manifold $M$ endowed with a Hessian metric $g$ and an affine flat connection $\con{0}$. A Riemannian metric $g$ on a manifold is said to be {\it Hessian} if in terms of an affine coordinates system with respect to $\con{0}$ (that means an affine local coordinate system $(x^1 ,..., x^n)$  around a given point such that $ \con{0}(dx^i)=0$) the metric is expressed as:
\[g=\sum_{i,j}\frac{\partial^2 \Phi}{\partial x^i \partial x^j}dx^i dx^j,\]
where $\Phi$ is a local smooth real-valued function. The metric tensor is symmetric since $\Phi$ is a smooth function (this is an application of the generalised Schwartz Theorem). The function $\Phi$ is called a potential.

\subsection{Codazzi tensors}\label{S:Codazzi}
We elaborate on the construction of Hessian manifolds and discuss Codazzi tensors.

Investigations on Codazzi tensors are important since Gauss--Codazzi equations form a sufficient and necessary condition for having the integrability of a given system (see \cite{Schouten}). This leads us in the end to Frobenius manifolds.

\subsubsection{}
A major object related to those Hessian structures are the Codazzi tensors.  Let $SEnd_{TM}$ be a vector bundle, where $TM$ is the tangent bundle of $M$, which consists of smooth symmetric (0,2) tensor fields of $TM$. The space of smooth sections of $SEnd_{TM}$ is denoted $S(M)$. The tensor $\kappa \in S(M)$ is a Codazzi tensor if:

\[\nabla_X\kappa(Y,Z)=\nabla_Y\kappa(X,Z),\] where $X,Y,Z \in TM$ and $\nabla_Y\kappa$ is a covariant derivative. The set of Codazzi tensors forms a vector bundle over $M$. 

\subsubsection{} For example, the Riemannian metric $g$ fulfils the Codazzi condition above, since one satisfies:

 \[\nabla_Xg(Y,Z)=\nabla_Yg(X,Z),\] where $X,Y,Z \in TM$. In particular, if $g$ is Hessian then this implies the existence of a 
 symmetric rank three tensor $A=(XYZ)\Phi$, where in local coordinates $X=\partial/\partial y^i, Y=\partial/\partial y^j, Z=\partial/\partial y^k$. The vector fields $X,Y,Z$ are called flat vector fields. The rank three symmetric tensor satisfies $A=\nabla_Xg(Y,Z)$ for $g(Y,Z)=(YZ)\Phi$.

If $(M,g)$ is a simply connected manifold of constant sectional curvature $K$ (possibly vanishing) a Codazzi tensor $\kappa$ can be given as: 

\begin{equation}\label{E:Hess}
\kappa= Hess(f) + Kgf, 
\end{equation}

where $f:M \to R$ is a smooth function.  

Conversely, on a manifold of constant sectional curvature $K$, any smooth function generates a Codazzi tensor via the formula Eq.~\ref{E:Hess}.

\subsection{Generalised Codazzi tensors}
Using the notation $\partial_{a}$ for the operator $\frac{\partial}{\partial x^a}$ (where $(x^a)$ are some local coordinates) the components of the metric tensor are written as
$g_{ij}=\partial_i \partial_j\Phi.$

\subsubsection{} As mentioned earlier, the metric tensor is a Codazzi tensor. In particular, for vector fields $X,Y,Z$ on $M$, $A(X,Y,Z)=\nabla_X  g(Y,Z)$ forms a totally symmetric (0,3)-tensor. We have $A(X,Y,Z)= A(X,Z,Y)=A(Y,X,Z)=A(Y,Z,X)=A(Z,Y,X)=A(Z,X,Y)$ by the Codazzi relation $\nabla_{X} g(Y,Z)=\nabla_{Y} g(X,Z)$ and because $g$ is symmetric.  

\subsubsection{} In local coordinates $(x^a)$, $(a=1,\cdots, n)$, the (0,3)-tensor  $\nabla_X  g(Y,Z)$ can be written as $\partial_k g_{ij}$ and by Codazzi's equation we have $\partial_k g_{ij}=\partial_j g_{ik}$, where $X=\partial_k$ and $Y=\partial_j$. If we assume that $g$ is Hessian, then the Codazzi equality $\partial_k g_{ij}=\partial_j g_{ik}$ becomes $\partial_k \partial_i \partial_j \Phi= \partial_j \partial_i \partial_k \Phi$. 
We denote  $\partial_k \partial_i \partial_j \Phi$ simply by $\Phi_{kij}.$

\subsubsection{Higher order Codazzi tensors}
An order $k$ Codazzi tensor is a smooth section $A$ of the vector bundle of $(0,k)$-tensor fields over $M$ satisfying the following equation:
\begin{equation}\label{E:CodGen}
    \nabla_{X_0}A(X_1,X_2,\cdots, X_k)=\nabla_{X_1}A(X_0,X_2,\cdots, X_k ) 
\end{equation}
for any tangent vector fields $X_0, X_1, X_2,\cdots, X_k$ on $M$.  

Since $\Phi$ is by construction a smooth function, we get $\Phi\in C^{\infty}(M)$. This allows the following generalisation to any order $k$ Codazzi tensor. If $A(X_1,\cdots, X_{k})=(X_1\cdots X_{k})\Phi$, where $\Phi$ is $C^\infty$
then the equation Eq. \eqref{E:CodGen} holds. We have thus  order $k$ Codazzi tensors, which are clearly totally symmetric tensors, by the smoothness assumption of $\Phi.$  

For example, it is easy to show that the rank three symmetric tensors $A(X,Y,Z)=(XYZ)\Phi$ satisfy the order three Codazzi relation
\[ \nabla_X A(X,Y,Z)= \nabla_Y A(X,Z,W),\] for vector fields $X,Y,Z$ and $W$ on $M$.

\subsection{On the associativity Equation}\label{S:AE} We survey briefly the Associativity Equation and recent progress related to it.
The Associativity Equation originates in the problem of finding a quasi-homogeneous function $\Phi$ with variables at $x=(x^1,\cdots,x^n)$ such that its third derivatives (for any $x$) are structure constants of an associative algebra $A_x$, with an $x$-independent unity. This is also known as the WDVV equation, named after Witten--Dijkgraaf--Verlinde--Verlinde. The Associativity Equations (or WDVV equations) are given by the following PDE:

\begin{equation}
 \forall a,b,c,d: \quad 
\sum_{e,f}\quad \partial_a\partial_b\partial_e\Phi ~ g^{ef}~ \partial_f \partial_c\partial_d\Phi= \sum_{e,f} \quad \partial_b\partial_c\partial_e\Phi ~ g^{ef}~ \partial_f\partial_a\partial_d\Phi. 
\end{equation} 
where, $\partial_{a}$ stands for the operator $\frac{\partial}{\partial x^a}$ for  some local coordinates $(x^a)$; $g^{ef}$ from the  contravariant components of a  non-degenerate Riemannian metric $g$ and $\Phi$ is a smooth function. The covariant components of $g$ are denoted $g_{ef}$. For the non-degenerate matrix $[g_{ef}]$ corresponding to $g$, the following relation $[g^{ef}]=[g_{ef}]^{-1}$ holds. 
As a result of the {\it geometrization} of the Associativity Equation emerged the notion of {\it Frobenius manifolds} (see \cite[p.20]{Man99}). 

In the physical context, the solutions of the WDVV equations express the moduli space of topological conformal field theories. 
They play a crucial role in the formulation of mirror symmetry for Calabi--Yau 3-folds. Particular solutions of the WDVV with certain properties are generating functions for the Gromov--Witten invariants of Kähler and symplectic manifolds.

This topic represent the culmination of several different mathematical paths  \cite{CoMa,CoVa,Dub96,Fei,Hit,KaKP,KoMa,LoMa,Man05,Man98,KonoMa,Sa,To04} (to cite only a few examples).  For instance: 
\begin{enumerate}
\item One has {\it formal} Frobenius manifolds~\cite{KoMa,KaKP}. Quantum cohomology is an example of such manifolds. In this formal case, formal solutions to the Associativity Equations are the same as cyclic algebras over the homology operad $H_*(\overline{M}_{0,n+1})$ of the moduli spaces of $n$-pointed stable curves of genus zero.
\item Isomonodromic deformations~\cite{Sa} in the framework of Saito's space and Hodge structures. 

\item The continuation of the school of Dubrovin--Novikov (investigations relying on integrable systems, bi-Hamiltonian systems, and equation of hydrodynamical type): \cite{BN,CoVa,Dub96,DN,Fei,Mokh95,KonoMa}. 

\item Recent developments on Frobenius manifolds have led Kontsevich to introduce the notion of $F$-bundles \cite{Kont22}. 
\end{enumerate}

\subsection{Associativity Equation $\&$ Codazzi tensors}

In this section, we show that Hessian manifolds with vanishing curvature exhibit an important geometric structure that allows solving the Associativity Equation, in terms of Codazzi tensors.

\begin{lem}\label{L:1.5.1}
Let $(M,g,\con{0})$ be a Hessian manifold. If $M$ is flat (i.e. has vanishing curvature)  then the Associativity Equation is satisfied.   
\end{lem}
\begin{proof}
 For the Hessian manifold $M$, the curvature tensor $R_{ijkl}$ is given by: 
\[R_{abcd}=\partial_c\partial_lg_{ab}-\sum_{ef} g^{ef} \partial_c g_{af} \partial_d g_{eb}. \]
If the Riemannian curvature of $M$ is zero, then it implies that 
$$\partial_k\partial_l g_{ab}=\sum_{e,f} g^{ef} \partial_k g_{af} \partial_l g_{eb}.$$ 

By assumption $g$ is Hessian, so we have that $\partial_c\partial_dg_{ab}=\partial_c\partial_d \partial_a\partial_b\Phi= \Phi_{cdab}$. 

Since the tensor $\Phi_{cdab}$ is fully symmetric, $\Phi_{cdab}=\Phi_{acdb}$. 
By hypothesis, the curvature tensor is null. Therefore,  
$\Phi_{acdb}=\sum_{e,f} g^{ef} \partial_b g_{af} \partial_d g_{ec}$.
The equality $\Phi_{cdab}=\Phi_{acdb}$ implies thus that  

$$\sum_{e,f} g^{ef} \partial_c g_{af} \partial_d g_{eb}=\sum_{e,f} g^{ef} \partial_b g_{af} \partial_d g_{ec}.$$ This corresponds to the Associativity Equation.

\end{proof}

We show the following statement. 
\begin{prop}\label{P:Codazzi1}
Let $(M,g,\con{0})$ be a Hessian manifold, satisfying the Associativity Equation. Then the  following equation is satisfied:

\begin{equation}\label{E:Codazzi}
    \sum_f \Phi^f_{ab}\con{0}_f g_{cd}=  \sum_f  \Phi^f_{bc}\con{0}_f g_{ad}.
\end{equation}
\end{prop}
\begin{proof}
Consider the space of rank three symmetric tensors.  Let $\Phi_{ijk}$ be a (0,3)-symmetric covariant tensor.  
Using partial dualization, we build the (1,2)-tensors via the metric $g$.  That is: 
\begin{equation}\label{E:Dual}
\sum_{k} \Phi_{ijk}g^{kf}=\Phi^f_{ij},
\end{equation} 
where the metric tensor components are given by $[g^{ef}]=[g_{ef}]^{-1}$. 
This operation provides a mixed tensor, one times contravariant and two times covariant.

Given $(M,g,\con{0})$ a Hessian manifold, the Associativity Equation can be expressed in terms of Codazzi tensors of order 2 and 3 as in Eq.\eqref{E:Codazzi}.

Indeed,  recall that the Associativity Equation is usually expressed as follows:    $$\forall a,b,c,d: \quad 
\sum_{e,f}\quad \partial_a\partial_b\partial_e\Phi ~ g^{ef}~ \partial_f \partial_c\partial_d\Phi= \sum_{e,f} \quad \partial_b\partial_c\partial_e\Phi ~ g^{ef}~ \partial_f\partial_a\partial_d\Phi.$$

From Eq.\eqref{E:Dual}, we can replace $\sum_e\Phi_{abe}g^{ef}$ by $\Phi^f_{ab}$. The symmetric rank 3 tensor  $\partial_f \partial_c\partial_d\Phi$ can be expressed as $\con{0}_f g_{cd}$. We apply the same rules for the right hand side of the equality. That is $\sum_e\partial_b\partial_c\partial_e\Phi ~ g^{ef}$ is expressed as $\Phi^f_{bc}$
and  $\partial_f\partial_a\partial_d\Phi$  is expressed as $\con{0}_f g_{ad}$. This generates the equation Eq.~\eqref{E:Codazzi}.
 \end{proof}

\section{(Pre-)Frobenius structures}\label{S:2}

\smallskip

\subsection{Pre-Frobenius manifold} 
We have discussed the case of affine Hessian manifolds and shown that the Associativity Equation can be satisfied if those manifolds are flat. While the Associativity Equation remains an object expressed in the language of differential geometry in \cite{Dub96}, a more algebraic version can be given. We explain below this approach by starting with the preliminary structure of {\it pre-Frobenius manifolds.} Then, the transition from pre-Frobenius manifolds to Frobenius manifolds is stated. 

\subsubsection{Pre-Frobenius structures}

 Let us recall the definition of a pre-Frobenius structure, from \cite[p.~18-19]{Man99}. The following construction is established for classical categories of manifolds: $C^\infty$, real analytic. The tangent sheaf is denoted $\cT_M$. 

A potential pre-Frobenius manifold consists of the following data (i)--(v).
\begin{enumerate}[label=(\roman*)]\label{Enum}
\item  A Riemannian manifold $M$ is equipped with an affine flat structure.  
The existence of such an affine structure implies the existence of a flat affine connection $\con{0}$ on the tangent space, and reciprocally. Therefore, on $M$ there exists a  flat (affine) torsionless connection $\con{0}$. 
\item[] 
    \item There exists a non-degenerate symmetric bilinear form $g$ on $\cT_M$, where $g$ is compatible with the connection $\con{0}$. 
       \item[] 
       
    \item  There exists a rank three symmetric tensor $A$ such that $A(X,Y,Z)=(XYZ)\Phi$, where $X,Y,Z$ are flat vector fields. In local flat coordinates, denoted $(x^a)$, one has the following: 
$A_{abc}=\partial_a\partial_b\partial_c \Phi,$ where $\Phi$ is a real potential function and the vector fields are $X=\partial_a, Y= \partial_b, Z =\partial_c$.

\item[]
 
\item There exists a bilinear composition law $\circ: \cT_{M} \times \cT_{M} \to \cT_{M} $.  This is given by the partial dualization $A\cdot g^{-1}$, providing a (1,2)-tensor field. The components of these tensor fields are given, in local coordinates, by   $\partial_a\circ \partial_b =\sum_{c}A_{ab}^c\partial_c,$ where $A_{ab}^c=\sum_eA_{abe}g^{ec}$.

\item[]

\item The rank three symmetric tensor $A$ satisfies for flat vector fields $X,Y,Z$ the relation $A(X,Y,Z)=g(X\circ Y,Z)=g(X,Y\circ Z)$.

\end{enumerate}

The following lemma allows to understand the geometric properties relating the multiplication and the covariant derivative.
\begin{lem}
Consider the affine space of connections on $M$. The difference of two connections $\nabla_1-\nabla_2$ coincides with a section of $\Omega^1_M \otimes End( \cT_{M})$ where $\Omega^1_M$ is the sheaf of 1-forms on $M$.  One has the following identification $\Omega^1_M \otimes End( \cT_{M})\cong Hom(\cT_{M}\otimes\cT_{M},\cT_{M}).$ 

Reciprocally, given a connection $\con{0}$ on $\cT_{M}$ and a section $\mathcal{A}$ of $Hom(\cT_{M}\otimes\cT_{M},\cT_{M})$ one can build a pencil of connections:

\begin{equation} \label{E:Connections}
\con{t}_X(Y)=\con{0}_X(Y)+t \mathcal{A},\end{equation}
where $t$ is a real parameter. 
The section $\mathcal{A}$ generates a multiplication operation $X\circ ^\mathcal{A} Y $ on $\cT_{M}$.

\end{lem}
\begin{proof}
The proof follows from basic properties of affine differential geometry. 
\end{proof}
Throughout the rest of the paper we shall write for simplicity $\circ$ instead of $\circ^\mathcal{A}$. 

\subsubsection{Identity vector field}
A vector field $e$ is called the identity if $e\circ X=X$ for all $X$.
If $A$ has a potential, then for all flat vector fields $X, Y$, 
$eXY\Phi=g(X,Y)$. 
If $e=\sum\limits_q \varepsilon^q\partial_q$ is a vector field where $\partial_q$ is flat then we have that:
\[A(e,X,Y)=\sum_q\varepsilon^q\partial_qXY\Phi=eXY\Phi=g(X,Y).\]

\subsubsection{Euler vector fields}\label{S:Euler}
A convenient way of expressing the quasi-homogeneity of the potential function $\Phi(x)$  is given in terms of an Euler vector field, that we define as being $E=\sum_a E^a(x)\partial_a$. A function is said to be quasi-homogeneous if $ \Phi(c^{m_1}x^{1}, \cdots, c^{m_n}x^{n})=c^{m}\Phi(x^1,\cdots, x^n)$  where $c$, $m$ and $m_i$ are non-null constants. 

 Translating the quasi-homogeneity in terms of Euler vector fields leads to having $Lie_E(\Phi(x))=m\Phi(x).$ The Euler vector fields are expressed locally as $E(x)= \sum_i (a_{j}^ix^j+b^i )\partial_i$, which shows that these are {\it affine vector fields} (see discussion in Proposition \ref{P:Euler}).

More globally if $(M,g,A)$ a pre-Frobenius manifold the Euler vector field satisfies for all vector fields $X,Y$ the following relation: 
\begin{equation}\label{E:Euler} [E,X\circ Y]-[E,X]\circ Y - X\circ [E,Y]= \beta X\circ Y\end{equation}
where $\beta$ is a constant (see \cite[p.24]{Man99}).

\subsection{Frobenius manifolds}

Recall that a {\it Frobenius algebra} over a field of characteristic 0 is a unital, commutative, associative algebra (finite dimensional) equipped with a non-degenerate symmetric bilinear form $\langle-,-\rangle$ where:  $\langle \, x\circ\,  y,z\rangle=\langle x,y\, \circ\,  z\rangle$, for all $x,y,z$ lying in the Frobenius algebra. 

According to \cite{Dub96} the  manifold $M$ admits the structure of a Frobenius manifold if:
\begin{itemize}
\item at any point of $M$ the tangent space has the structure of a Frobenius algebra; 
\item the invariant inner product $\langle-,-\rangle $ is a flat metric on $M$;
\item the unity vector field satisfied $\con{0}e=0$;
\item The rank 4 tensor $(\nabla_W A) (X,Y,Z)$ is fully symmetric;
\item The vector field $E$ is determined on M such that $\con{0}(\con{0}E)=0$. 

\end{itemize}

The Euler vector field is clearly an affine vector field c.f. Prop.\ref{P:Euler}. Its existence is tightly related to the fact that we have an affine structure. Therefore, this definition can be encapsulated in the following more condensed. 

\begin{dfn}\label{D:preF+F}  
Assume $M$ is a potential pre-Frobenius manifold (resp. domain) satisfying conditions (i)--(v). If  $(\cT_M,\circ)$ is a Frobenius algebra, then $M$ is a Frobenius manifold (resp. locus).
\end{dfn}

In particular, a pre-Frobenius manifold $M$ is Frobenius if and only if the  connections $\con{t}$ defined in Eq. \eqref{E:Connections} are flat (see \cite[Theorem 1.5]{Man99}). This is in fact equivalent to satisfying the Associativity Equations.

\begin{rem}
One can check that the Euler characteristic of a Frobenius manifold is  0. 
\end{rem}

\section{Monge-Ampère manifolds are potential pre-Frobenius manifolds}\label{S:PDE}

\subsection{Monge--Ampère domains and geometrization of the elliptic Monge--Ampère equation}

\subsubsection{Geometrization of an elliptic Monge--Ampère equation (GEMA)}
If $\D$ is a strictly convex bounded subset of $\mathbb{R}^n$ then for any nonnegative function $f$ on $\D$ and continuous $\tilde{g}:\partial \D \to \mathbb{R}^n$ there is a unique convex smooth function $\Phi\in C^{\infty}(\mathscr{D})$ such that 
\begin{equation}\label{E:EMA}
\det \mathrm{Hess}(\Phi)= f, 
\end{equation} in $\D$ and $\Phi=\tilde{g}$ on $\partial \D,$ (see \cite{RT} Eq. (1.1) and Eq. (1.2))

\medskip

The {\it geometrization of an elliptic Monge--Ampère equation} refers to the geometric data generated by $(\mathscr{D}, \Phi)$, where \begin{itemize}
    \item $\mathscr{D}$ is a strictly convex domain 
    \item $\Phi$ a real convex smooth function (with arbitrary and smooth boundary values of $\Phi$)  
    \end{itemize}
    such that Eq.~\eqref{E:EMA} is satisfied.   For simplicity, we use the symbol $(\mathscr{D},\Phi)$ to refer to a GEMA.

\subsubsection{Codazzi tensors and EMA}
Positive definite Codazzi tensors can be determined by elementary symmetric functions of eigenvalues $P_s(k)$  where, $s=1,2,...,n$. If $k$ is symmetric, then it has $n$ real eigenvalues $l_{1},\cdots, l_{n}$, roots of the equation $\det(k-l g)=0$. If the sectional curvature differs from zero, then the elementary symmetric polynomial $P_n(k)$ is given by $$P_n(k)=\frac{\det(Hess(f)+Kgf)}{\det(g)}$$ and this forms an equation of  Monge-Amp\`ere type. In the case where $k$ is positive definite, this provides elliptic solutions.  

\subsubsection{Affine structure on GEMA}
As a more general approach, we consider a domain $\mathscr{D}$ being a GEMA (instead of a Monge-Ampère manifold). 
\begin{lem}\label{L:Aff}
 $(\mathscr{D},\Phi)$ is equipped with an affine flat structure.  
\end{lem}
\begin{proof}

Eq.~\eqref{E:EMA} is invariant under unimodular linear transformations of the variables $(x^1,\cdots, x^n)\in \R^n$ (see~\cite{Cal}). This coincides with the definition stated in Sec.\ref{S:2.2.1}. Therefore, $\mathscr{D}$ has an affine flat structure. 
\end{proof}

\begin{cor}\label{C:con}
$(\mathscr{D},\Phi)$ is equipped with flat and torsionless affine connection $\con{0}$.
\end{cor}
\begin{proof}
By Lemma \ref{L:Aff}, we may assume that the domain $\mathscr{D}$ is equipped with an affine flat structure. This is equivalent to the existence of a flat, torsionless affine connection $\con{0}$ on this domain.
Therefore, $(\mathscr{D},\Phi)$ is equipped with a flat torsionless affine connection, as expected.
\end{proof}

\subsubsection{Hessian structure of GEMA}
In the following, our focus goes mainly to tensors defined by second and third derivatives. In particular, this implies the construction of existence of a Riemannian metric, which is Hessian.

\begin{lem}\label{L:rank2,3}
$(\mathscr{D},\Phi)$ is equipped with a Hessian metric $g$ and a rank three symmetric tensor $A$, which are defined in local coordinates as follows:

\begin{equation}\label{E:EQ1}
g_{ij}(x)=\frac{\partial^2\Phi}{\partial x^i \partial x^j},
\end{equation}

\begin{equation}\label{E:EQ2}
A_{ijk}(x)=\frac{\partial^3\Phi}{\partial x^i \partial x^j\partial x^k}, 
\end{equation}
where $i,\, j\ ,k\, =1,2,\cdots, n$.
\end{lem}
\begin{proof}
By Lemma~\ref{L:Aff}, $\mathscr{D}$ has an affine flat structure. If one allows only linear transformations of the variables $(x^1,\cdots, x^n)\in \R^n$,
then the array of all partial derivatives of $\Phi$ of any given order can be interpreted as the components of a covariant tensor, symmetric in all pairs of indices. Therefore, $g$ and $A$ are symmetric tensors. 
By assumption $\Phi$ is convex. Therefore, the isometric tensor defined in Eq.~\eqref{E:EQ1} is positive definite. 
The domain of $\mathscr{D}$ is thus equipped with a Riemannian metric $$ds^2=g_{ij}dx^idx^j.$$ 
Given that the metric tensor is expressed as a Hessian of $\Phi$ (see Eq.~\eqref{E:EQ1}) 
the Riemannian metric $g$ is a Hessian metric.
Moreover, to each component $g_{ij}$ one can associate naturally the contravariant tensor $g^{ij}$. 
Therefore, we have proved that there exists a non-degenerate Hessian metric $g$ and a symmetric rank three tensor $A$ on the domain $\mathscr{D}$.
\end{proof}

\subsubsection{} Whenever it is clear from context that the domain $\mathscr{D}$ is a manifold we will call the object $(\mathscr{D},\Phi)$ a Monge--Ampère manifold. This latter coincides with the notion defined in \cite[Sec. 3.2]{KoS01}. We show in Sec. \ref{S:GEMA} that a Monge--Ampère manifold forms a pre-Frobenius manifold.

\subsection{Pre-Frobenius structures on GEMA}\label{S:GEMA}
We now prove that the quintuple $(\mathscr{D},\Phi,g,A, \con{0})$ has a pre-Frobenius structure. 
\begin{thm}\label{T:main}
Let $\mathscr{D}$ be a domain (resp. manifold). 
The quintuple $(\mathscr{D},\Phi,g,A, \con{0})$ is a potential pre-Frobenius domain (resp. manifold).
\end{thm}

\begin{proof}

\begin{enumerate}[label=\alph*)]
    \item By Lemma~\ref{L:Aff} and Corollary \ref{C:con} the domain $\mathscr{D}$ has an affine flat structure. Therefore, this satisfies  the statement (i) in the above definition.

 \item[]
 \item By Lemma~\ref{L:rank2,3}, there exists a symmetric (non degenerate) bilinear form $g$.  In local coordinates, $g_{ij}$ is by Eq.~\eqref{E:EQ1}. Therefore, this satisfies  the statement (ii) in the above definition. This metric is compatible with the connection $\con{0}$. 

 \item[]
\item There exists a symmetric (non degenerate) rank three tensor $A$, defined in Eq.~\eqref{E:EQ2}. Therefore, this satisfies  the statement (iii) in the above definition. 

 \item[]
\item  The multiplication operation $\circ$ on $\cT_M$ can be defined using the covariant derivative. In particular, we have $\nabla_{X}(Y):=X\circ Y.$  In local coordinates, this is defined as: $\partial_a\circ \partial_b =\sum_{c}A_{ab}^c\partial_c,$ where $A_{ab}^c=\sum_eA_{abe}g^{ec}$ and $g^{ec}$ is the contravariant. So, (iv) is satisfied. 

 \item[]
\item  Let us check the equality in (v). In local coordinates: 
\end{enumerate}
\[g(\partial_a \circ \partial_b,\partial_c)= g(\sum_{e}A_{ab}^e\partial_e,\partial_c)= \sum_{e}A_{ab}^eg(\partial_e,\partial_c) =\sum_{e}\sum_{f} A_{abf} g^{fe} g_{ec}=A_{abc}= \partial_a\partial_b\partial_c\Phi.\]

 On the other side, 
 \[g(\partial_a, \partial_b \circ\partial_c)=g(\partial_a,\sum_{e}A_{bc}^e\partial_e)=\sum_{e}A_{bc}^eg(\partial_a,\partial_e)=\sum_{e}\sum_{f} A_{bcf} g^{fe} g_{ea}=  A_{bca}=\partial_b\partial_c\partial_a\Phi.\]
 By symmetry of $A$, we have $\partial_b\partial_c\partial_a\Phi=\partial_a\partial_b\partial_c\Phi$.
Therefore,  $g(\partial_a \circ \partial_b,\partial_c)= g(\partial_a, \partial_b \circ\partial_c)=\partial_a\partial_b\partial_c\Phi.$
 In other words, for flat vector fields this is: 
 $$A(X,Y,Z)=g(X\circ Y,Z)=g(X,Y\circ Z).$$
 Therefore, we have demonstrated that  $(\mathscr{D},\Phi,g,A, \con{0})$ provides a pre-Frobenius structure on $\mathscr{D}$. 
\end{proof}

In particular, an immediate corollary is that:

\begin{cor}
    The Monge--Ampère manifold forms a pre-Frobenius manifold.
\end{cor}

\subsection{Frobenius loci and Associativity Equations} We discuss the possibility that $(\mathscr{D},\Phi,g,A, \con{0})$ contains a locus satisfying the Associativity Equations. Such a locus is called a {\it Frobenius locus}.

\begin{thm}\label{T:preF}
Let  $(\mathscr{D}, \Phi, g, A, \con{0})$ be the geometrization of EMA, where $\mathscr{D}$ is a strictly convex bounded domain and $\Phi$ is a smooth convex function with smooth boundary values.
Assume that it contains a non-empty flat locus $\mathscr{N}$. Then, $\mathscr{N}$ is a Frobenius locus and the WDVV equation is satisfied on $\mathscr{N}$.
\end{thm}
\begin{proof}
By construction $\mathscr{N}\subset \mathscr{D}$ inherits a pre-Frobenius structure. 
By hypothesis, $\mathscr{N}$ is flat. 

Let us compute the covariant derivative $A_{acd;b}$ of the tensor $A_{acd}$. Recall that $\sum_e g^{ef}\cdot A_{ecd}=A_{cd}^f$. The covariant derivative $A_{acd;b}$ of the rank three symmetric tensor is given by the following formula: 
 \begin{equation}\label{}
A_{acd;b}=-\frac{1}{2}\frac{\partial^4 \Phi}{\partial x^a\partial  x^c\partial x^d \partial x^b}+ \sum_{e,f} g^{ef}(A_{ecd}A_{fab}+A_{aed}A_{fcb}+A_{ace}A_{fdb}).
\end{equation}

This is symmetric in all pairs of indices. 
From this, one can verify that the Riemann curvature tensor can be expressed in terms of the rank three symmetric tensors as follows:
 \begin{equation}\label{E:C}
 R_{acdb}=\sum_{e,f} g^{ef}(A_{eab}A_{fcd}-A_{ead}A_{fcb}).
 \end{equation}
By the flatness hypothesis on $\mathscr{N}$, the curvature tensor vanishes on $\mathscr{N}$. Therefore, from Eq.\ref{E:C} it follows that we obtain the equality  $\sum_{e,f} g^{ef}A_{eab}A_{fcd}=\sum_{e,f} g^{ef}A_{ead}A_{fcb}.$ 
Naturally this can be expressed as $\sum_{e,f} A_{eab}g^{ef}A_{fcd}=\sum_{e,f} A_{ead}g^{ef}A_{fcb}.$  
By symmetry of the rank 3 symmetric tensors, we can rewrite this as $\sum_{e,f} A_{abe}g^{ef}A_{fcd}= \sum_{e,f} A_{fcb}g^{ef}A_{ead}$.   
Therefore, it follows that on $\mathscr{N}$ the following holds:
 \begin{equation} \forall a,b,c,d: \quad 
\sum_{e,f}\partial_a\partial_b\partial_e\Phi g^{ef}\partial_f \partial_c\partial_d\Phi= \sum_{e,f} \partial_f\partial_c\partial_b\Phi g^{ef}\partial_e\partial_a\partial_d\Phi. 
\end{equation}
In other words, those are the Associativity Equations. Therefore, if there exists a flat pre-Frobenius locus $\mathscr{N}$ in a pre-Frobenius domain then it satisfies  the WDVV equation and forms therefore a Frobenius locus. \end{proof}

Therefore, this statement allows to establish geometrically a relation between the elliptic Monge--Amp\`ere equation and the WDVV equation. 

\begin{prop}\label{P:Codazzi2}
 Let $\mathscr{N}$ be an $m$-dimensional submanifold of an $(m+k)$-dimensional domain $\mathscr{D}$. Suppose that there exists an isometric immersion $f:\mathscr{N}\to \mathscr{D}$. The domain $\mathscr{D}$ lies in the Euclidean space such that $\mathscr{D}$ is a GEMA. 
The submanifold $\mathscr{N}$ is Frobenius if and only if the Riemannian curvature tensor of $\mathscr{D}$ is expressed as: \[R(X,Y,Z,W)=g(\alpha(X,Z),\alpha(Y,W))-g(\alpha(Y,Z),\alpha(X,W))\]
where $\alpha$ is a second fundamental form for the given immersion and $X,Y,Z,W$ are vector fields on $\mathscr{N}$.
\end{prop}
\begin{proof}
 Assume that $\mathscr{N}$ is a Frobenius manifold. Then, by construction, $\mathscr{N}$ is flat. Let $X$ and $Y$ be vector fields on $\mathscr{N}$. At every point in $\mathscr{N}$, one defines $\tilde{\nabla}_X(Y)$ the tangential component and by $\alpha(X,Y)$ its normal component. One has:
 \[\nabla_{X}Y=\tilde{\nabla}_{X}Y+\alpha(X,Y),\]
 where $\tilde{\nabla}$ is a covariant derivative of a Levi--Civita connection on $\mathscr{N}$ and $\alpha(X,Y)$, is a differentiable field of normal vectors to $\mathscr{N}$.

 By Gauss equations, we have the following relation
\[R(W, Z, X, Y) = R(W, Z, X, Y) + g(\alpha(X, Z), \alpha(Y, W))- g(\alpha(Y, Z),\alpha(X, W)),\]
where $X, Y, Z,$ and $W$ are arbitrary tangent vectors to $\mathscr{N}$, $R$ is the Riemannian curvature tensor of $\mathscr{D}$ and $R'$ is the Riemannian curvature tensor of $\mathscr{N}$. Since $\mathscr{N}$ is flat, $R(W, Z, X, Y) =0$. Hence the relation $R(W, Z, X, Y) = g(\alpha(X, Z), \alpha(Y, W))- g(\alpha(Y, Z),\alpha(X, W))$ for vector fields $X, Y, Z,$ and $W$ on  $\mathscr{N}$.

Reciprocally, if $R(W, Z, X, Y) = g(\alpha(X, Z), \alpha(Y, W))- g(\alpha(Y, Z),\alpha(X, W))$ for vector fields $X, Y, Z,$ and $W$ on  $\mathscr{N}$ then it means that $R=0$. This exactly means that $\mathscr{N}$ is flat. By Theorem \ref{T:preF} it satisfies the WDVV equation. So, it is a Frobenius manifold. 
 
\end{proof}


 \section{Landau--Ginzburg theory and its relations to Frobenius manifolds}\label{S:LG1}
We examine the theory proposed by Landau and Ginzburg to explain the phenomenon of superconductivity, in the light of Frobenius manifolds. This theory remains interesting for the community of solid state and elementary particle physicists. Appendix~\ref{S:LG} presents an introduction to LG theory.  In \cite{Cecotti,LW,Chiodo}, LG models are applied in the context of Saito spaces and Calabi--Yau manifolds. We stick to the original LG construction to prove that the Hilbert space coming from the LG theory is parametrised by a Monge-Ampère domain. This step serves as an in-between step for proving our final statement. We explain how the Monge--Ampère manifold can be obtained.

\subsection{Monge problem}
Originally, the Monge problem was to find the optimal transport $T:\R^d\to \R^d$ between two distribution of masses $\rho_1$ in a domain $A\subset \R^d$ to a distribution $\rho_2$
on a domain $B\subset \R^d$ such that $$\int_A \rho_1 dx=\int_B \rho_2 dx$$ and the total mass remained preserved under transport. 
\subsection{Measures}
Let $(\Omega,\mathcal{F},P)$ be a  measure space. A measure $P$ is said to be absolutely continuous w.r.t the measure $\lambda$ if for every measurable set $C$, the equality $\lambda(C)=0$ implies $P(C)=0.$  The measurable function is given by $P(C)=\int_C  \rho d\lambda,$ $\forall C \subset \mathcal{F}$, where $\rho$ is called the density of the measure $P$ and $\rho= \frac{dP}{d\lambda}$ is the Radon--Nikodym derivative.

\subsection{Monge--Ampère}\label{S:8.3}
Consider the (probability) measure space $(\R^d,\mathcal{F},\mu)$, where $\mathcal{F}$ is formed from Borel sets in $\R^d$ and $\mu$ is a Borel measure.
By $\mathcal{P}(\R^d)$ we denote the space of Borel probability measures on $\R^d$.

Let $T:\R^d \to \R^d$ be a Borel transformation defined $\mu$-almost everywhere such that if the measure of a measurable set $A\subset \mathcal{K}\subset \R^d$ (where $\mathcal{K}$ is a compact)  vanishes, then the measure of $T^{-1}(A)$ vanishes too. Given a Borel set $M \subset \R^d$ one can {\it pushforward} $\mu$ through $T$, resulting in a Borel probability measure $T_{\diamond}\mu$ on $\R^d$ given by: 
\begin{equation}\label{E:nu}
    T_{\, \diamond\, }\mu[M]:=\mu[T^{-1}(M)].
\end{equation}

By~\cite{Bre1,Bre2}, there exists a unique map $h$ such that $$T=\nabla U\circ h,$$ where $h$ preserves the volumes and $U$ is a convex Lipschitz function in the neighborhood of $\mathcal{K}$.

Obtaining the Monge-Ampère equation requires a few additional steps. Let $V^*$ be the Legendre transform of $U$ that is $V=U^*$. By Brenier's factorisation theorem~\cite{Bre1}, $\nabla U$ and $T$ map $\mathcal{K}$ into the same set $\mathcal{K}^*$ that is 
$T(\mathcal{K})=\nabla U(\mathcal{K})=\mathcal{K}^*$ and $\nabla V(\mathcal{K}^*)=\mathcal{K}.$
For any continuous function $f\in C^1$ on $\mathcal{K}$ we have the following:

\begin{equation}\label{E:MA-Brenier}
    \int_\mathcal{K} f(T(x))dx=\int_\mathcal{K} f(\nabla U(x))dx=\int_{\mathcal{K}^*} f(y)\det (D^2 V(y))dy
\end{equation}
where we have the change of variables $y=\nabla u(x)\in \mathcal{K}^*$ and $x=\nabla V(y)\in \mathcal{K}$.
One deduces the existence of a positive function $r(x)$, integrable in $\mathcal{K}^*$ such that 
\[\int_\mathcal{K} f(T(x))dx=\int_{\mathcal{K}^*} f(y)r(y)dy.\]
Therefore, we have that $V$ is a solution to the Monge--Ampère problem. The function $V$ is convex and we have $\det D^2 V(y)=r(y)$ almost everywhere on $\mathcal{K}^*.$

\subsection{Discretisation}
\label{S:Conf} 
We discuss a discrete version of the situation above. Consider a given set of $m$ objects in $\R^d$ all having the same mass and a  Monge-Ampère transport. This allows us to work with configuration spaces of $m$ marked points in $\R^d$.

\smallskip 
Given a smooth topological space $X$ and a positive integer $m$, define $X[m]:=X^m\setminus \Delta$ to be the $m$-th (ordered) configuration space of $X$, where $X^m$ is the Cartesian product of $m$ copies of $X$ (equipped with product topology) and $\Delta$ is the fat diagonal generated by $x_i=x_j$. This is the set of $m$-tuples of pairwise distinct points in $X$. 
\begin{lem}\label{L:Borel}
 To each configuration of $X[m]$ corresponds a counting measure defined on a Borel $\sigma$-algebra.   
\end{lem}

 \begin{proof}
    Assume $(X,\mathcal{U})$ is the topological space $X$ endowed with its topology $\mathcal{U}$. Then, the open sets $U_{\alpha}$ of the topology $\mathcal{U}$ generate a $\sigma$-algebra.
    This forms thus a {\it Borel} $\sigma$-algebra. On each Borel set we count how many marked points exist. This can be obtained by using a counting measure, which is a Borel measure.
 \end{proof}

So, for a configuration space $X[m]$, where $X=\R^d$ we can consider the measure space $(\Omega,\mathcal{F},\mu)=(\R^d,\mathcal{F},\mu)$ where:

\begin{itemize}
    \item $\mathcal{F}$ is the $\sigma$-algebra generated by Borel sets of $\R^d$. 
    \item Given a Borel set $M$, the measure $\mu$ is taken to be the  counting measure given by the number of marked points in $M$ divided by the total number of marked points (here $m$).
\end{itemize}

\begin{prop}\label{P:Conf}
Let $X=\R^d$ and $X[m]$ the configuration space of $m$ marked points on $\R^d$. 
Then, there exists a class of paths in the configuration space $X[m]$ can be parametrised by a transport $T$. 
\end{prop}

\begin{proof}
Take two different configurations in $X[m]$. They correspond to a pair of distinct points, say {\bf x} resp. {\bf y}. Such a pair of points can be related by a path in the space of configurations. 

Given that each configuration in $X[m]$ corresponds to a Borel probability measure on $\R^d$ (cf. Lemma \ref{L:Borel}), we can push forward the measure $\mu$ (illustrating the configuration {\bf x}) to another Borel measure $\nu$. This is done by using Eq. \eqref{E:nu}. 

Therefore, we can define an optimal transport which parameterizes a path in $X[m]$. 
\end{proof}

\subsection{Landau--Ginzburg theory à la Koopman-von Neumann theory versus Landau--Ginzburg model}\label{S:psi} 
Let us discuss the LG theory versus the LG model.  

\medskip 

An LG model is defined as a pair $(X,W)$, where $X$ is a non-compact K\"ahler manifold and $W:X\to \C$ is a holomorphic Morse function, called superpotential. Identifying the superpotential $W$ with a quasi homogeneous polynomial, the zero locus $X_W=\{W=0\}$ defines a hypersurface in a weighted projective space. Some additional conditions imply that the hypersurface $X_W$ is Calabi--Yau.

\medskip 

The LG theory viewed from the Koopman-von Neumann perspective {\it implies} the LG model. However, the converse is not true. An explication of this fact is given below. An LG model captures a {\it local} aspect of the LG--KvN theory. However, it omits important structures existing in the LG theory. For instance, the LG--KvN theory is based on considerations of wave functions (also known as order parameters) and free energy which, as such, is {\it not} explicitly mentioned in the LG model. 

\medskip 

This wave function, in the LG theory, being an element of a Hilbert space $\mathfrak{H}$ has a very important property, which is fundamental for our construction. These wave functions are complex-valued square-integrable functions, defined over a phase space $\mathcal{M}$. In particular, 
if $x$ is a point in $\mathcal{M}$ then $\psi(x)$ is a wave function defined at $x$, where the coordinates of $x$ are given by $x=(q^1,\cdots,q^N,p^1,\cdots p^N)$ ; $q^i$ are the position and $p^i$ the momentum.
von Neumann 
\medskip 

The square of the absolute value of a wave function $\psi$ (i.e. as the amplitude multiplied with its own complex conjugate) denoted $|\psi|^2=\psi\overline{\psi}$ is crucial for expressing the free energy formula of LG theory. Furthermore, $\rho(\psi):=|\psi|^2$ has a meaning as a probability density. In particular, the function $\rho(\psi)$ is a positive semi-definite function. As $\psi(x)$ is square-integrable it turns that $\rho(\psi)$ is integrable (i.e. $L^1$).

\medskip 

\begin{itemize}

 \item[\ding{169}] Given $\psi\in \mathfrak{H}$ we have a normalisation of the Hilbert space in the sense that $\int_{\mathcal{M}} |\psi|^2 d\lambda=1$. This condition implies that $|\psi|^2=\psi\overline{\psi}$ is a {\it density} of probability;  the  function $\psi$ defines a probability measure $P$: $$P(C)=\int_C|\psi|^2 \, d\lambda$$ where $C\in\mathcal{F}$ and $|\psi|^2$ is the Radon--Nikodym derivative $\frac{dP}{d\lambda}$.
    
 \end{itemize}

Via the Koopman--von Neumann's construction, it is postulated that there is an evolution of the wave function $\psi(x,t)$ with respect to a real parameter $t$, given by the Liouville equation. Note that the Liouville equation is satisfied as well by the classical probability density $\rho(\psi)$:
\[\imath \frac{\partial\rho(\psi,t)}{\partial t}=\widehat{L}\rho(\psi,t),\]
where 
$\widehat{L}=\imath \frac{\partial \widehat{H}}{\partial q^i}\frac{\partial}{\partial p^i}-\frac{\partial \widehat{H}}{\partial p^i}\frac{\partial}{\partial q^i}$ and $\widehat{H}$ is a Hamiltonian operator of the LG theory. 

\begin{itemize}
    \item[\ding{169}] The Hermitianity of $\widehat{L}$ is a necessary and sufficient condition to guarantee the unitarity of the evolution operator,  given by 
    \begin{equation}\label{E:Ut}
         U_t=\exp{-\imath\widehat{L}t}
    \end{equation}        

\end{itemize}

\medskip 

  Let $\mathcal{M}$ be a variety (a phase space) and consider the group of automorphisms $S_t$ (where $t$ is an affine parameter) having the positive integral invariant 
  
  $\int \rho d\lambda$, where $\rho$ is a positive, single valued analytic function on $\mathcal{M}$ (the integral is taken over an arbitrary region on $\mathcal{M}$). 
  
  Let $\psi(x)$ be a complex-valued function at a point $x$ in $\mathcal{M}$ such that: $\psi$ is single-valued; $\psi$ is measurable; the Lebesgue integrals $\int_{\mathcal{M}} (d^{2n}x) |\psi| d\lambda$ and $\int_{\mathcal{M}} (d^{2n}x) |\psi|^2 d\lambda$ are finite.

The totality of such functions $\psi$ form a Hilbert space $\mathfrak{H}$, the metric of which is given by the inner product $\langle \psi,\varphi\rangle= \int_{\mathcal{M}}(d^{2n}x) \psi(x)\overline{\varphi}(x)$.

One defines the one-parameter group of unitary transformation maps $U_t$ on $ \mathfrak{H}$, continuous in $\psi$ as given by  
$$U_t\, \psi(x)=\psi(S_t\, x),$$ for all real $t$. The transformation $U_t$ is a unitary, linear endomorphism of $\mathfrak{H}$ such that:
$$U_t(a\psi(x)+ b \varphi(x))=aU_t(\psi(x))+ b U_t( \varphi(x))$$ and $\langle U_t\psi,U_t\varphi\rangle =\langle \psi,\varphi\rangle. $

For any $\psi \in \mathfrak{H}$, we have 

$$\bigg[\frac{\partial}{\partial t}U_t\psi(x)\bigg]_{t=0}=\imath {\bf P} \psi(x),$$ where ${\bf P}$ is a self-adjoint linear operator on $\mathfrak{H}$.
\medskip 

To construct the Monge--Ampère operator/domain we need to investigate the spacial evolution of the wave function $\psi$. To do this, we look for the free energy of the LG theory. That is one more reason to consider LG theory instead of the LG model.

\begin{rem}
The presented theory above can be viewed as the Koopman--von Neumann theory. The Koopman--von Neumann construction \cite{Koo,vN}, implies that the wave functions lie in a Hilbert space $\mathfrak{H}$. Axioms of Koopman--von Neumann imply that the wave functions lie in the Hilbert space of square integrable functions, with respect to a density function over a phase space. The probability density on phase space is expressed in terms of the underlying wave function, in the Hilbert space $\mathfrak{H}$.
\end{rem}

\begin{itemize}
 \item[\ding{169}]  On the other side,  the property Eq.\eqref{E:Ut} guarantees the {\it conservation of the total probability.} 
\end{itemize}
These structures presented above, form the basics for the construction of the Monge--Ampère domain in Theorem~\ref{T:LG}, Corollary~\ref{T:END} and Theorem~\ref{T:FINAL}.

\medskip 

\begin{itemize}
\item[\ding{169}] Following from the fact that  $\psi\in \mathfrak{H}$ and $\exp{\imath \theta}\, \psi\in \mathfrak{H}$ belong to the {\it same equivalence class} we have a torus action on $\mathfrak{H}$ given by $\psi(x) \mapsto \exp{\imath \theta} \, \psi(x)$, where $\theta\in \R$. In particular, this leads us to work in the framework of weighted projective spaces. A weighted projective space is given by $$\mathbb{P}_{(w_1,\cdots,w_n)}=\bigg(\C^{n+1}\setminus \{0\}/\C^{\times}\bigg),$$ where $(w_1,\cdots,w_n)$ are the weights and the weighted action is given by

 $g \cdot (x_1,\cdots, x_n)\mapsto (g^{w_1}x_1,\cdots, g^{w_N}x_n)$ for $g\in \C^{\times}$.

\end{itemize}

\subsubsection{}\label{S:Chiodo}
Using the results in \cite{Vafa,Chiodo,CR}, we go back  to the LG model. 
According to \cite[Sec.6]{Vafa}, one can ``characterize a large class of Calabi--Yau manifolds using the fixed points of the Landau-Ginzburg theory, under renormalization group flows''. One relates the superpotential of the Landau--Ginzburg model to defining equations of Calabi--Yau manifolds in a weighted projective space.  

Given that we have established that $(\mathfrak{H},\mathscr{Y},\pi)$ generates weighted projective spaces (see the previous construction) we choose to rely on~\cite{BH,K} to construct mirror pairs of Calabi--Yau manifolds/orbifolds. 
Let us define a hypersurface inside this weighted projective space. It is defined via a quasi homogeneous polynomial $W$ in the variables $x_1,\cdots,x_n$ such that 

$$W(s^{l_1}x_1,\cdots,s^{l_n}x_n)=
s^dW(x_1,\cdots,x_n),$$ where, $s\in \C^{\times}$; 
$l_i$ 
are positive rational numbers satisfying $l_i=w_i/d$ with $gcd(w_1,\cdots,w_n,d)=1$  and where $w_i$ are the {\it weights} of the projective space.  
Then, $$X_W:=\bigg( \{W=0\}_{ {\C^{n+1} \setminus \{0\} }} /\C^\times  \bigg)$$ defines a quotient stack lying in a weighted projective space. We require $W=0$ to have at most a unique isolate singularity at 0 and that it satisfies the Calabi--Yau condition i.e. the sum  $\sum \limits_{i=1}^nl_i=1$, implying that  $X_W$ is Calabi--Yau. It is naturally possible to omit the singular locus (or the corresponding divisor $D$ with normal crossings)  by considering $X_W\setminus D$.

Let $Aut(X_W)$ be the group of diagonal symmetries leaving $W$ invariant i.e. given an element $(a_1,\cdots, a_n)\in Aut(X_W)$ we have the following equality $$W(a_1x_1,\cdots,a_nx_n)= W(x_1,\cdots,x_n).$$ Consider $J$ a subgroup of the automorphism group $Aut(W)$ defined by \break
 $J=(\exp{2\pi\imath l_1},\cdots, \exp{2\pi\imath l_n})$. Then, $J$ has a trivial action on $X_W$.  One considers the group ${\bf G}$ given by the quotient {\bf G}=$G/J$, acting  faithfully on $X_W$, where $G$ is a subgroup of $Aut(X_W)$ containing $J$ and such that $J\subset G\subset SL_N(\C)$. The last requirement implies that $X_W/{\bf G}$ is Calabi--Yau.

By the Berglund–Hübsch–Krawitz approach, one can construct a {\it mirror dual } of such a Calabi-Yau manifold, say $X_{W'}$, forming thus a mirror pair. In particular, the quasi homogeneous polynomial $W'$ of this mirror dual is obtained by {\it transposing} the exponents matrix (see \cite{K} for the details).  Note that there are a few restrictions on $W$ such as being invertible and non-degenerate. An invertible potential is non-degenerate if and only if it can be written as a sum of (decoupled) invertible potentials called atomic types. There are three such types. Defined in $\C^N$, they can be expressed as: 

\begin{itemize}
\item[] \[W_\text{Fermat} = x^m.\]
\item[] \[W_\text{loop}= x_1^{m_1}x_2+x_2^{m_2}x_3+\dotsb +x_{N-1}^{m_{N-1}}x_N+x_N^{m_N}x_1.\]
\item[] \[W_\text{chain}= x_1^{m_1}x_2+x_2^{m_2}x_3+\dotsb +x_{N-1}^{m_{N-1}}x_N+x_N^{m_N}.\]
\end{itemize}

One obtains thus another quasi homogeneous polynomial $W'$ defining a hypersurface lying in another weighted projective space. The group {\bf G'} satisfies analogous relations as its dual.  For more details, see \cite{BH,K}.  By Theorem 4~\cite{CR}, the Calabi Yau $X_W/${\bf G} and the Calabi--Yau  $X_{W'}/${\bf G'} form a mirror pair.

\subsection{Duality}

\begin{thm}\label{T:LG}
Let $\mathfrak{H}$ be the Hilbert space of square integrable functions  defined by the  Landau--Ginzburg/Koopman--von Neumann theory.
Then, there exists a real Monge--Ampère domain $\mathscr{Y}$ parametrising $\mathfrak{H}$, so that $\pi:\mathfrak{H} \to \mathscr{Y}$ forms a torus fibration. 

\end{thm}

 \begin{proof}

\begin{itemize}
    \item[(1)] We first build the base space of the torus fibration. 
\end{itemize}

To do so we rely on Sec.\ref{S:psi}. Let $\psi(x) \in \mathfrak{H}$. By the normalisation property, one obtains a density of probability $\rho(\psi)=|\psi|^2$ with associated measure $P(C)=\int_{C}\rho d\lambda$, where $C\subset\mathcal{F}$ is a Borel set. By construction, $P$ is absolutely continuous with respect to the Lebesgue measure. 

\medskip 

To each complex valued-function lying in $\mathfrak{H}$ one can associate a density of probability with corresponding measure $\mu\in \mathcal{P}(\mathcal{M})$, where the measure lies in the space of absolutely continuous probability measures on $\mathcal{M}$. 
The space of probability distributions defined above forms our base space $\mathscr{Y}.$

The construction of this also implies the projection morphism $\pi$ from the space $\mathfrak{H}$ formed from the complex valued $L^2$ functions to the space of density of measures, denoted $\mathscr{Y}$. 

\begin{itemize}
    \item[(2)]  We now show that there exists a torus fiber bundle. 
\end{itemize}

Take any point of the base space $\rho(\psi)=(\rho_1,\cdots,\rho_n)(\psi) \in \mathscr{Y}$. By the above construction $\rho_i(\psi)=\psi_i\overline{\psi}_i$, where $\psi_i$ is a given wave function in $\mathfrak{H}$. Then, the fiber $\pi^{-1}(\rho)$ is given by the following equivalence class $(\psi_1,\cdots,\psi_n)\sim (\exp{\imath \theta_1}\, \psi_1,\cdots,\exp{\imath \theta_n}\, \psi_n)$ where $\theta_i\in \R$. This means that if $\psi\in \mathfrak{H}_{\rho}$ is a fiber  then  $y\psi\in \mathfrak{H}_{\rho} $ for $y\in \mathfrak{T}$, where $\mathfrak{T}$ is a torus.
Therefore, this illustrates a torus fiberation.
To be more precise we have the following factorisation:

\[
\begin{tikzcd}
 & \tilde{\mathfrak{H}}  \arrow{dr}{\tilde{f}} \\
\mathfrak{H} \arrow{ur}{\tilde{\pi}} \arrow{rr}{\pi} && \mathscr{Y}
\end{tikzcd}
\]
where $\tilde{\mathfrak{H}}$ is the quotient space of $\mathfrak{H}$ by the torus  $\mathfrak{T}$. The triple $(\mathfrak{H},\tilde{\mathfrak{H}},\tilde{\pi}) $ forms a $\mathfrak{T}$-principal bundle.

\begin{itemize}
    \item[(3)]  We  show that the base space $\mathscr{Y}$ is a Monge--Ampère domain.
\end{itemize}

Following Sec. \ref{S:8.3}, construct a transport map $T$ (irrotational and without crossings). The transport map pushes forward the measure $P$ to a measure $Q$ by using the construction depicted in Eq.~\ref{E:nu} i.e. 
 $$Q=T_{\diamond}P[M]:=P[T^{-1}(M)],$$

 where $M$ is a Borel set of $\mathcal{M}$ and 

the measure $Q$ corresponds to another  vector $\widetilde{\psi}$ in the Hilbert space $\mathfrak{H}$.

Let $P,Q\in  \mathcal{P}(\mathcal{M})$. Assume that $P$ vanishes on Borel subsets of $\mathcal{M}$ having Hausdorff dimension $N-1$. Then, there exists a convex function $U$ on $\mathcal{M}$ such that the gradient of $U$ pushes forward $P$ to $Q$ (\cite{Bre1,Bre2}). Note that $U$ is not necessarily unique. However the map $\nabla U$ is uniquely determined $P-$almost everywhere.

If the unicity of $U$ is required, then it can be obtained under the following extra conditions. One needs that the first moment $\int|y|dQ(y)$ of $Q$ should be finite and that $P$ is almost continuous w.r.t Lebesgue (that is $dP(x)=f(x)dx$ with $f$ vanishing outside a bounded smooth connected domain $D\subset \mathcal{M}$ (and bounded away from 0 and infinity on $D$).

Following Caffarelli's construction~\cite{Caf1,Caf2}, if $dP(x)=f(x)dx$ and $dQ(x)=g(y)dy$ the conditions on $f$ and $g$ ensure the regularity of $U$. One obtains therefore the elliptic Monge--Ampère relation:
\begin{equation}
    det\bigg[ \frac{\partial^2 U}{\partial x_i \partial x_j}\bigg]g(\nabla U(x))=f(x)
\end{equation}

Therefore, we can define paths in the space of states which are parametrised by a Monge--Ampère operator.
Such a transport construction holds for any pair of probability distributions in $\mathscr{Y}$. Therefore, $\mathscr{Y}$ is a Monge--Ampère domain. Note that there are some possible singularities of the domain, provided by the singular fibers.  
    
 This ends the proof of our statement.    
    
\end{proof}
\begin{rem}As a side remark about the proof, note that in the Monge--Ampère construction one can use a displacement interpolation procedure. Given a parameter $t \in [0, 1]$, the interpolant $\mu_t \in P(\mathcal{M})$ between two measures $\mu$ and $\nu$ is defined as
$$\mu_t:=\mu\diamond [(1-t)id + t\nabla U].$$
\end{rem}

We show the relation from the LG model to Frobenius manifolds.  
\begin{cor}\label{T:END}
The base space $\mathscr{Y}$ of the torus fiberation $(\mathfrak{H},\mathscr{Y},\pi)$ is 
a real (potential) pre-Frobenius domain.
\end{cor}
\begin{proof}
 The first step is to apply Theorem \ref{T:LG}. Indeed, this shows that the LG theory is parametrized by a Monge--Ampère domain. The second step is to apply Theorem \ref{T:main}, where it is shown that a Monge--Ampère domain is a potential pre-Frobenius domain. Therefore, the Hilbert space $\mathfrak{H}$ of the LG  theory is parametrized by a real potential pre-Frobenius domain.
\end{proof}    

We can now show that mirror dual Calabi--Yau manifolds $X$ and $X^\vee$, obtained via the construction of Berglund--H\"ubsch--Kotwitz are parametrised by the same Monge--Ampère manifold, which is a space of density of probabilities. 

\begin{thm}\label{T:FINAL}
There exists a real Monge--Ampère domain $\mathscr{Y}$ parametrizing a pair of mirror dual Calabi--Yau manifolds $X/${\bf G} and $X'/${\bf G'}, obtained from the Berglund--H\"ubsch--Kotwitz approach. This construction forms Lagrangian torus fibrations over $\mathscr{Y}$. 
\end{thm}
\begin{proof}
\begin{enumerate}
\item We show that the mirror pairs of Calabi--Yau manifolds $X/${\bf G} and $X'/${\bf G'} generated by the zero locus of $W$ and $W'$ are parametrised by a unique Monge--Ampère domain. We refer to \cite{BH,K} and Theorem  \cite{CR} for the method of construction of mirror Calabi--Yau pairs.

Suppose by contradiction that there exist two distinct Monge--Ampère domains $\mathscr{Y}$ and $\mathscr{Y'}$ parametrising respectively the weighted projective spaces  containing those two hypersurfaces. 

First, from the construction of the proof of Theorem \ref{T:LG}, one can deduce that $\mathscr{Y}$ is independent from the of weights of the projective spaces containing the hypersurfaces.

Secondly, the assumption that there exist two distinct spaces $\mathscr{Y}$ and $\mathscr{Y'}$ implies that there would exist two different spaces of distributions of probabilities defined on the phase space $\mathcal{M}$. Given that those densities of probabilities are generated by the wave functions of $\mathfrak{H}$ it implies that the Hilbert space $\mathfrak{H}$ of $L^2$ functions would be formed from two disjoint Hilbert subspaces of wave functions having very different properties. This is absurd. So, $\mathscr{Y}$ and $\mathscr{Y'}$ coincide. 

~

\item We prove now that this construction generates a torus fibration. 
\end{enumerate}
We consider a bundle constructed from the triple $(\mathfrak{H},\mathscr{Y},\pi)$, where  $\mathfrak{H}$ is the  Hilbert space  of $L^2$ functions; $\mathscr{Y}$ is the Monge--Ampère domain corresponding to space of density of probabilities over the phase space $\mathcal{M}$, such as defined by Koopman--von Neumann. This generates a principal fiber bundle given by $\C^{n+1} \to  \C^{n+1}/\mathfrak{T}$ where $\mathfrak{T}$ is a torus. 

Take a point on $X_W$ of coordinates $(\psi_0(x),\cdots,\psi_n(x))\in X_W\subset \C^{n+1},$ where $x\in \mathcal{M}$. This point of the total space is mapped to a point in the base space i.e. lying in the space of density of probabilities, with coordinates $\rho=(\rho_0,\cdots,\rho_n)(\psi(x))\in \R^{n+1} $, where $\rho_i=|\psi_i(x)|^2$. The fiber $L_0=\pi^{-1}(\rho(\psi))$ generates a torus via the fact that  $\pi^{-1}(\rho(\psi))\simeq (\C^\times)^{n+1}\times \{(\psi_0(x),\cdots,\psi_n(x))\}$  (recall that $\psi_i$ lies in the same equivalence class as $\exp{\imath \theta}\psi_i$).  The same type of construction holds for the mirror dual, given by $\{W'=0\}$. We refer to \cite{K} for a detailed exposition. 

We have thus the following diagram:

\[\begin{tikzcd}[column sep=tiny]
& X_W/{\bf G}   \ar[dl, "\pi_1"] 
&
&[1.5em] \\
  \mathscr{Y} 
    &
      & \C^{n+1}\setminus 0 /\C^\times  \ar[dl]  \ar[ul]
& \ar[l, dashed, ]\mathfrak{H} \ar[dll, "j_2"', bend left=20] \ar[ull, "j_1"', bend right=20]\\
& X_{W'}/{\bf G'} \ar[ul, "\pi_2"] 
&
&
\end{tikzcd}\]

In the Monge--Ampère domain $\mathscr{Y}$ we transport
the density of probabilities corresponding to those wave functions, whose complex values satisfy $\{W=0\}$ in a small neighborhood of $x$ in $\mathcal{M}$ to wave functions such that the complex values $\{W'=0\}$. This is resumed in the following diagram:
\[
\begin{tikzcd}
&\mathfrak{H}\arrow[d,dashed] \ar[ddddl,, bend right=60] \ar[ddddr, bend left=60]& \\
&\C^{n+1}\setminus 0 /\C^\times&\\
\arrow[hook]{ur}{} X_W\subset \C^n \arrow{rr}{} \arrow[swap]{d}{fiber\, \times(\C^\times)^n }& & \arrow{ll}{} X_{W'} \arrow[hook]{ul}{} \subset\C^n\arrow{d}{fiber\, \times(\C^\times)^n }{} \\
X_W/{\bf G } \arrow[rr,dashed,"Mirror\, dual"] \arrow{d}{\pi_1}  && X_{W'}/{\bf G '} \arrow{d}{\pi_2} \\
\rho\in \mathscr{Y} \arrow{rr}{Monge\, transport\, T}  & &\arrow{ll}\tilde{\rho}\in \mathscr{Y} 
\end{tikzcd}
\]

\end{proof}

\begin{rem}
Let $\rho$ be a point on the Monge--Ampère domain. Given a smooth fiber $\pi^{-1}(\rho)=L_0$ and a collection of loops $(\gamma_1,\cdots \gamma_n)$ forming a basis of the first cohomology group $H_1(L_0,\mathbb{Z})$ one can determine an affine (Monge-Ampère) chart, centered at $\rho$. Indeed, given a point in the base space $\rho\in \mathscr{Y}$ and a point $\tilde{\rho}$ in a small neighborhhod of $\rho$  producing a smooth fiber $L_1=\pi^{-1}(\tilde{\rho})$, one can isotope $L_0$ to $L_1$ among fibers of $\pi$. This isotopy is parametrised by a Monge--Ampère operator.  
Each loop $\gamma_i$ traces a cylinder $\Gamma_i$ with boundary in $L_0\cup L_1$. The affine coordinates associated to $\rho$ are the symplectic areas $(\int_{\Gamma_1}\omega,\cdots, \int_{\Gamma_n} \omega)$ where $\omega$ is the Kähler form on the considered (open) Calabi--Yau manifold lying in a weighted projective space.
 \end{rem}

We highlight that other objects discussed in \cite[Sec. 2]{AAK} such as the {\it phase function, "instanton corrections'',  tautologically unobstructed Lagrangians} and {\it walls} in the Monge--Ampere domain can be easily recovered, using our approach. In particular, the walls in the Monge--Ampere domain  will be naturally expressed in terms of probabilities. It will be the subject of a next article.

\section{LG Toy model}\label{S:ToyModel1}
 \subsection{}
 The state space of an $n$-dimensional quantum system is the set of all $n \times n$ positive semidefinite complex matrices of trace 1, known as {\it density matrices}. We consider as a toy model the space of such matrices. For simplicity, we do not normalise the matrices (i.e. the trace is not required necessarily to be equal to 1) and omit the boundary of the cone given by $x^TAx=0$, where $A$ is a symmetric square matrix.  
 
Those matrices play an important role in the LG theory (see Sec.\ref{S:LG1} and Sec.\ref{S:LG}). It is possible to identify the rank one matrices, corresponding to pure states, with nonzero vectors in a complex Hilbert space of dimension $n$. Notice that the same state is described by a vector $\psi$ and $\kappa \psi$, where $\kappa \neq 0$, which leads to projective geometry. 
In this paper, we do not consider the infinite-dimensional case, where density matrices are replaced by density operators and the space of pure states is the complex projective space over the infinite-dimensional Hilbert space.

The open cone of positive definite complex matrices forms a strictly convex self dual homogeneous cone. Given a Hilbert space $\mathfrak{H}$, one defines a self dual cone $\mathscr{P}$ in $\mathfrak{H}$ by the set  $\mathscr{P}=\{ \zeta\in \mathfrak{H} \,|\, \langle \zeta,\eta \rangle \geq 0, \forall \eta \in \mathscr{P} \}$. We know from the classical representation theorem of Jordan--von Neumann--Wigner that there are five classes of irreducible formally real Jordan algebras. The transitively homogeneous self dual cone associated to a given class of Jordan algebras is then the set of positive elements of the Jordan algebra, with the Hilbert structure given by the natural trace~\cite{Kos62,Maa,Vin,PS,CoMa}. In fact, by \cite{Connes}, the category of von Neumann algebras is equivalent to the category of self dual facially homogeneous complex cones. We refer to Appendix \ref{A:1} for details concerning those cones.

  \subsubsection{}
Let $\K$ be a real division algebra (i.e.  $\K$  is  $\R, \C, \hH$ or in $\oO$). The objects considered here are irreducible cones, represented as the space of positive definite symmetric $n\times n$ matrices with coefficients in $\K$. Note that case of octonions differs from the others in the sense that we only have the cone $\mathscr{P}_3(\oO)$ formed from $3\times 3$ matrices with coefficients in $\oO$. Whenever there is no source of confusion, we write $\mathscr{P}$ rather than  $\mathscr{P}_n(\K)$ to designate any cone depicted above. 

\subsection{Main statement for the Toy Model}

The aim of this section (and the following) is to prove the following fact. 
\begin{thm-non}
Let $\mathscr{P}$ be an irreducible strictly convex cone, defined as previously. Then:
\begin{enumerate}[label=\arabic*)]
\item  $\mathscr{P}$ is a non-compact symmetric space.  
\item $\mathscr{P}$ carries a pre-Frobenius structure. 
\item There exists a non-empty Frobenius locus $\F$ in $\mathscr{P}$, given by $\F=\exp{\frak{a}}$, where $\frak{a}\subset \ft$ is a Lie triple system i.e. satisfies $[[\frak{a},\frak{a}],\frak{a}]\subseteq \frak{a}$.
\end{enumerate}
\end{thm-non}

\begin{itemize}
\item[\ding{169}] Statement 1) follows from a well known classification of symmetric spaces given by Nomizu~\cite{No54}. We will recall related facts in Sec. \ref{S:3.1}. 

\item[\ding{169}]  Statement 2) has two different proofs. 
\begin{itemize}
\item The first proof relies on the method of elliptic Monge--Ampère equations, developed in Sec.\ref{S:GEMA}. It is given by  Proposition~\ref{P:pdsmPreF}.
\item The second proof is the subject of Sec. \ref{S:Pre-Frob}), namely Proposition \ref{P:pre-Fro}. It uses  methods from \cite[Chap. IV]{Hel}. The latter approach allows to obtain Statement 3).
\end{itemize}
\item[\ding{169}]  Statement 3) is obtained from two main results:
\begin{itemize}
\item There exists a totally geodesic submanifolds in $\mathscr{P}_n$ (see Sec. \ref{S:tgLietriple}, Proposition \ref{P:exist} and Proposition \ref{P:Matrix}) 
\item There exists a totally geodesic submanifold in $\mathscr{P}_n$ carrying a Frobenius structure (see Sec. \ref{S:Frob}, Theorem~\ref{T:Frobenius}).
\end{itemize}
\end{itemize}

\subsubsection{} We proceed with a first proof of Statement 2).
\begin{prop}\label{P:pdsmPreF}
Let $\mathscr{P}$ be an open cone of positive definite symmetric  square  matrices, with coefficients in a real division algebra $\K$. Then, $\mathscr{P}$  carries a pre-Frobenius structure. \end{prop}

\begin{proof}
First note that those open cones can be identified with the space of positive definite quadratic forms in $n$ variables. Applying this, we can rely on  Gårding--hyperbolic polynomials defined on a vector space of quadratic forms on $\R^n$ (\cite{Ga}). 
For each homogeneous polynomial $P$ on this vector space one considers the associated (non- linear) partial differential operator defined by $P (\mathrm{Hess} (\Phi)).$ If $P$ is the determinant, the associated operator is the real Monge--Ampère operator.

By taking $P$ to be the determinant, one can enlarge the construction above to the classes corresponding to the division algebras, i.e. $\R,\C,\hH$ and $\oO$.
The construction for  real numbers is given in \cite{RT}. The complex case is investigated in \cite{BT}. The quaternion and  octonion cases have been considered in  \cite{Al,AV}. 

In fact, symmetric matrices with determinant greater or equal to $c>0$ form a convex set. To every such convex set corresponds an elliptic Monge--Amp\`ere operator, given by $\det \mathrm{Hess}(\Phi)=c>0$, where $\Phi$ is unique. One can see the details of the construction in \cite{Ga,BT}. In \cite{RT,Al,AV} a generalisation to other fields such as $\C, \hH$ and $\oO$ is given.

So, the domain $\mathscr{P}_n$ comes equipped with an elliptic Monge--Amp\`ere operator and is a GEMA. Therefore, by Theorem~\ref{T:preF} it has a pre-Frobenius structure. 
\end{proof}

\begin{rem} A method to check whether  $\mathscr{P}_n$  is a Frobenius manifold or contains one is to compute the curvature tensors.  From those calculations, it follows that the sectional curvature of  $\mathscr{P}_n$ is non-positive. Therefore  $\mathscr{P}_n$  is a pre-Frobenius domain. Since there exists a non-empty locus where the sectional curvature vanishes this implies the existence of a Frobenius locus in $\mathscr{P}_n$. We have thus an explicit Monge--Ampère domain containing a Frobenius manifold. 
\end{rem}
\subsection{An application to Calabi--Yau manifolds}

Consider Siegel's upper half plane $\mathscr{H}_n$. Given a positive integer $n$, this space is defined as 
$$\mathscr{H}_n:=\{Y\in Mat_{n\times n}(\C)\, |\, Y=Y^t,\, Im(Y)>0\, \},$$
where $Mat_{n\times n}(\C)$ denotes the set of all matrices of size $n\times n$ with entries in $\C$; the notation $Im(Y)>0$ means positive definite. 

Viewed as a symmetric space it is identified with the quotient of Lie groups $Sp_{2n}(\R)/U_n,$ where $Sp_{2n}(\R)$ is the symplectic group. The symmetric space $\mathscr{H}_n$ contains a symmetric subspace 
$$\mathscr{P}_n(\R)=\{Y \in Mat_{n\times n}(\R)\, | \, Y=Y^t > 0\}$$ the (open cone) of symmetric positive definite matrices. 

The space $\mathscr{H}_n$ parametrizes an important arithmetic variety: the moduli of principally polarized abelian varieties of dimension $n$, denoted $\mathscr{T}_n^\C$. On the other hand, $\mathscr{P}_n(\R)$ parametrizes principally polarized real tori of dimension $n$, by \cite{Ya15}.

Therefore, the Monge--Ampère domain $\mathscr{P}_n(\R)$ parametrizes the space $\mathscr{T}_n^\R$ of principally polarized real tori of dimension $n$ Indeed, to each equivalence class $[Y]\in GL(n, \mathbb{Z})\curvearrowright\mathscr{P}_n$, where $Y\in \mathscr{P}_n$ (and where the notation $\curvearrowright$ means ``acting on'') one associates a principally polarized real torus $T_Y=\R^n/\Lambda_Y$ , where $\Lambda_Y=Y \mathbb{Z}^n$ is a lattice in $\R^n$.

Naturally, we can summarise this in the following diagram:
\[\begin{tikzpicture}[every node/.style={midway}]
  \matrix[column sep={4em,between origins}, row sep={2em}] at (0,0) {
    \node(R) {$\mathscr{T}_n^\C$}  ; & \node(S) {$\mathscr{T}_n^\R$}; \\
    \node(R/I) {$\mathscr{H}_n$}; & \node (T) {$\mathscr{P}_n(\R)$};\\
  };
  \draw[<-] (R/I) -- (R) node[anchor=east]  {};
  \draw[->] (R) -- (S) node[anchor=south] {};
  \draw[->] (S) -- (T) node[anchor=west] {};
  \draw[->] (R/I) -- (T) node[anchor=north] {};
\end{tikzpicture}\]
where:
\begin{itemize}
    \item $\mathscr{H}_n$ is Siegel's upper half plane.
        \item $\mathscr{P}_n(\R)$ forms the imaginary part of $\mathscr{H}_n$. It is identified with real cone of symmetric positive definite matrices of size $n\times n$.
    \item $\mathscr{T}_n^\R$ denotes the space of principally polarized real tori of dimension $n$ (polarized real abelian varieties).
    \item $\mathscr{T}_n^\C$ denotes the space of principally polarized complex tori of dimension $n$ (polarized  abelian varieties).

\end{itemize}
This illustrates an example of \cite[Sec. 8]{KoS01}, in which complex tori are an example of Calabi-Yau manifolds. In particular, this provides a torus fibration. 

\medskip 

In the next sections we provide the building blocks for giving a geometric proof of the existence of a isometrically immersed Frobenius manifold in $\mathscr{P}$. This Frobenius manifold forms an {\it algebraic} torus. 


\subsection{Geometry of $\mathscr{P}$}
In Appendix \ref{A:1} we recall all the necessary information relating symmetric cones to quotients of Lie groups.

\subsection{Totally geodesic submanifolds of $\mathscr{P}$}
\subsubsection{Lie triple systems}\label{S:tgLietriple} We prove in this subsection that there exists in $\mathscr{P}_n(\K)$ a flat totally geodesic submanifold. 

\begin{thm}\label{T:Hel}
\begin{enumerate}~

    \item The curvature tensor $R$ evaluated at $T_p\mathscr{P}$ is given by 
\[R(X,Y)Z=-[[X,Y],Z],\, \text{for}\quad  X,Y,Z\, \in\,  T_p\mathscr{P}. \]

\item Consider the submanifold $\F\subset \mathscr{P}$. Let $p$ be a point in $\mathscr{P}$. Identifying the tangent space $T_p \mathscr{P}$ with $\ft$, let $\frak{a}\subseteq \ft$ be a Lie triple system contained in $\ft$. 
Put $\F:=\exp{\frak{a}}$. Then, $\F$ has a natural differentiable structure in which it is a  totally geodesic submanifold of $\mathscr{P}$ satisfying $T_p\mathscr{F}=\frak{a}$. On the other hand, if $\F$ is a totally geodesic submanifold of $\mathscr{P}$ then the subspace $\frak{a}=T_p\mathscr{F}$ of $\ft$ is a Lie triple system.  

\end{enumerate}
\end{thm}
\begin{proof}
This follows essentially from \cite[Thm IV.4.2. and Thm IV.7.2]{Hel}.
\end{proof}

Assuming that there exists such a totally geodesic submanifold in the symmetric space $\mathscr{P}$,  we focus on those submanifolds which are of the highest dimension. 
A {\it maximally $r$-flat totally geodesic submanifold} is a totally geodesic submanifold such that $r$ is the biggest natural number where the totally geodesic submanifold is isometric to $\mathbb{R}^r$.
\begin{rem} Assume $\frak{a}\subseteq \ft$ is a maximal abelian subspace of dimension $r$. Then by Theorem~\ref{T:Hel}, $\F=\exp{\frak{a}}$ is a maximally $r$-flat totally geodesic submanifold in the symmetric space $\mathscr{P}$.\end{rem}

\subsubsection{Maximally flat totally geodesic submanifolds} We show the explicit existence of flat submanifolds in symmetric cones $\mathscr{P}_n$. By Theorem~\ref{T:Hel}, if there exists a Lie triple system $\frak{a}\subset \frak{t}$ one has a totally geodesic submanifold  defined by $\F=\exp{\frak{a}}$. For classes described in Table~\ref{T:table1}, there exists a non-empty class of Lie triple systems. A class of such  Lie triple systems are given by Cartan subalgebras of $\frak{t}$. 

\begin{prop}\label{P:exist}
  Assume $\mathscr{P}$ is one of the non-compact symmetric spaces in Table \ref{T:table1}. There exist totally geodesic submanifolds given by 
     $ \F=\exp{\frak{a}},$
  where $\frak{a}\subset \fg$  is Cartan subalgebra. 
\end{prop}
\begin{proof}
Assume $\mathscr{P}$ is one of the non-compact symmetric spaces in Table \ref{T:table1}. 
Then the Lie algebra $\fg=\frak{gl}_n$ associated to $G=GL_n$ is a semisimple Lie algebra. 
There exists a class of Lie subalgebras $\frak{a}\subset\fg$  which are Cartan subalgebras of $\fg$. 
Cartan subalgebras are maximal abelian subalgebras of $\fg$ and form a Lie triple system.  The existence of totally geodesic submanifolds  is ensured by Theorem \ref{T:Hel}.
\end{proof}
\subsubsection{Classification of flat totally geodesic submanifolds in $\mathscr{P}_n$}
We prove algebraically that in $\mathscr{P}_n(\K)$ there exist totally geodesic submanifolds with vanishing sectional curvature.

\smallskip 

\begin{prop}\label{P:Matrix}
Assume $\mathscr{P}_n(\K)$ is a cone described in Table~\ref{T:table1}. 
For every such cone, there exists an $(n-1)$-dimensional totally geodesic submanifold $\F$ given by
 $\F:=\exp{\tilde{\frak{a}}},$ where ${\tilde{\frak{a}}}$ is the Cartan subalgebra of $\frak{gl}_n(\K)$ given by
${\tilde{\frak{a}}}=c I_n\oplus \frak{a}$, $c\in\K$ and $\frak{a}$ is described as follows:

\begin{table}[ht]
    \centering
   
    \begin{tabular}{|c|l|}
       \hline
    $\K$  & $\frak{a}$ is represented by \\
      & \\
    \hline
    $\R$   & \text{all diagonal matrices with real diagonal entries and such that the trace is 0} \\
    & \\
    
     $\C$    & \text{all diagonal matrices with diagonal entries $a+\imath b$ and such that the trace is 0}\\
       & \\
     $\hH$ & \text{ all diagonal matrices of} 
      $\left\{\left(\begin{smallmatrix}
        X& -\overline{Y}\\
        Y & \overline{X}
    \end{smallmatrix}\right)\, \mid \, \Re Tr X=0 \right\}$\\
      & \\
    $\oO$ & \text{all diagonal $(3\times 3)$ matrices with diagonal entries $a+\imath b$} \\
    &  \text{and diagonal matrices of a Cartan subalgebra of} $\frak{g}_{2}$
    \\
      & \\
      \hline
    \end{tabular}
    \caption{Some Cartan subalgebras}
    \label{T:Cartansubalgebra}
\end{table}
\end{prop}
\begin{proof}
By  \cite[Thm IV.7.2 ]{Hel}, totally geodesic submanifolds are of the form $\exp{\frak{a}}$, where $\frak{a}\subset \ft$ is a Lie triple system. 

The non-compact symmetric spaces considered here (Table \ref{T:table1}), lead to restricting attention only to the case of semi-simple Lie algebras. The question is to investigate the existence of Cartan subalgebras. 
In the case considered, those are maximal abelian subalgebras. The Cartan involution is given by $ X\mapsto -X^t$.

For real division algebras $\K$, the Cartan subalgebras of $\frak{sl}_n(\K)$ are  outlined below. 
\begin{enumerate}
    \item Assume $\K=\R$. The Lie algebra attached to the symmetric space has the Cartan decomposition $$\frak{sl}_n(\R)=\frak{so}_n\oplus \frak{s}ym_0(n),$$ where $\frak{s}ym_0(n)$ is the set of symmetric matrices of trace 0 with entries in $\R$. The maximal abelian subspace $\frak{a}$ of $\frak{s}ym_0(n)$ is given by the set of diagonal matrices of null trace.
Therefore, the maximally flat totally geodesic submanifold is given by  
\[\F=\exp{\frak{a}}=\{Diag(\lambda_1,\cdots, \lambda_n): \,  \lambda_i=\exp(t_i)\in\, \R,\,  \prod_{i=1}^n \lambda_i=1\}.\]

  \item Assume $\K=\C$. The Lie algebra $\frak{a}$ is given by diagonal matrices with complex entries. Therefore, the maximally flat geodesic submanifold is
\[\F=\exp{\frak{a}}=\{Diag(\lambda_1,\cdots, \lambda_n): \,  \lambda_i=\exp({a_i+\imath b_i})\in\, \C,\, \prod_{i=1}^n \lambda_i=1.\}\]

  \item  Assume  $\K=\hH$. The Lie algebra  $\frak{sl}_n(\hH)$ is given by diagonal matrices in $$\left\{\left(\begin{smallmatrix}
        X& -\overline{Y}\\
        Y & \overline{X}
    \end{smallmatrix}\right)\, |\, \Re Tr X=0 \right\}.$$ To obtain $\frak{a}$, take  the diagonal matrices in that set.

  \item Assume $\K=\oO$. The Lie algebra  $\frak{a}$ is given by all  diagonal $(3\times 3)$ matrices with diagonal entries $a+\imath b$ and diagonal matrices of a Cartan subalgebra of $\frak{g}_{2}$.
\end{enumerate}
So, we have given explicit examples of Lie triple systems $\frak{a}$, which generate totally geodesic submanifolds in $\mathscr{P}$. 
\end{proof}

\begin{lem}
The tangent space $T_p\F$ carries the structure of a commutative associative, unital algebra.
\end{lem}
\begin{proof}
The algebra $\tilde{\frak{a}}$ corresponds to $T_p\F$.
The algebra $\tilde{\frak{a}}$ is generated by diagonal matrices. So,  this generates an associative, commutative and unital algebra on $T_p\F$ for the usual matrix multiplication product.\end{proof}
\section{Pre-Frobenius cone and algebraic torus}\label{S:Pre-Frob}
We give a second proof that the cones $\mathscr{P}$ carry a pre-Frobenius structure and find a Frobenius manifold isometrically immersed in $\mathscr{P}$. This approach does not rely on elliptic Monge--Ampère equations. It allows a more detailed study of the geometry of the cones since we construct explicitly the metric, the rank 3 symmetric tensor and the potential function. 

\subsection{}
Let $\mathscr{A}$ be an $n$-dimensional  algebra over $\K$ (possibly with unit ${\bf 1}_\mathscr{A}$). Denote the basis by $\{e_i\}_{i=1}^n$. The structure constants  $C^i_{jk}$ are components of the (1,2)-tensor\, $\circ: \mathscr{A} \times \mathscr{A} \to \mathscr{A}$ such that: $e_i\circ e_j=\sum_s C_{ij}^s e_s,\quad i,j,s\in \{1,\cdots, n\}.$

\medskip 
\begin{itemize}
    \item  If $\mathscr{A}$ is associative then 
\begin{equation}\label{E:Assoc}
    \sum_c C_{ab}^cC_{cd}^f=    \sum_c C_{ad}^cC_{cb}^f.
\end{equation} 
    \item  If $\mathscr{A}$ is commutative then  \begin{equation}\label{E:comm} C^s_{ab}=C^s_{ba}.\end{equation}
\end{itemize}

\subsection{Koszul--Vinberg algebra}\label{S:KV}

By Koszul \cite{Kos59,Kos61,Kos62,Kos68A,Kos68B,Kos85}, important relations can be established between the space of connections on an affine flat manifold and the Koszul--Vinberg (KV) algebra, also known as a pre-Lie algebra. 

Let us introduce a multiplication operation $\circ$ on the tangent sheaf $\cT_M$, where $M$ is an affine flat manifold $M$
such that: \begin{equation}\label{E:Koszul} 
X\circ Y:=\nabla_X(Y), \end{equation} where $X,Y$ are vector fields. This algebraic structure $(\cT_M,\circ)$ forms a KV algebra. 

Consider $X,Y$ two vector fields in $\cT_M$. If $M$ carries an affine flat structure then the following relations are satisfied:
 \begin{equation}
\nabla_X(Y)-\nabla_Y(X)-[X,Y]=0,
\end{equation}
also, written as:
\[\nabla_X\nabla_Y-\nabla_Y\nabla_X-\nabla_{[X,Y]}=0,\]
where $X,Y\in \cT_M$.

One can easily check that the  multiplication operation $``\circ"$  given by $X\circ Y=\nabla_X(Y)$ defines a commutative algebra satisfying the relation:

\[a\circ (b\circ c)- b\circ (a\circ c) = (a\circ b)\circ c- (b\circ a)\circ c,\]
for $a,b,c\in T_M$

 An algebra $(\mathscr{A},\circ)$ is a {\it pre-Lie algebra (or Koszul--Vinberg KV algebra)} if for all $ a,b,c\in \mathscr{A},$ we have $(a \, \circ \,  b) \,  \circ \,  c-a \, \circ \,  (b \, \circ  \, c)=(b \, \circ \,  a) \, \circ  \, c-b \, \circ  \, (a \, \circ \,  c).$

This is also known as a Lie-admissible algebra. Each Lie algebra with affine structure is derived from a Lie-admissible algebra. 

\smallskip 

We illustrate this below in the context of $\mathscr{P}$.

\begin{dfn}\label{P:FA}
Let $\mathscr{P}$ be as in Table~\ref{T:table1}.
For any point $x\in \mathscr{P}$, the tangent sheaf $(T\mathscr{P},\circ)$ forms a  pre-Lie algebra, where  the multiplication is given in local affine coordinates as follows:

\begin{equation}\label{E:circ}
(X\circ Y)^i=-\sum_{j,k}\Gamma^{i}_{jk}(x)X^jY^k\quad 1\leq i\leq n,\end{equation}
where:
\begin{itemize}
\item $X,Y\in T_\mathscr{P}$ are vector fields; 
    \item the structure constants of the algebra are given by $\Gamma^{i}_{jk}$, where $\Gamma^{i}_{jk}=\frac{1}{2}\sum_l\partial_{jkl}\Phi g^{li}$ 
and $\Phi$ is a potential function;
     
    \item $\partial_{jkl}\Phi=A_{jkl}$ forms a rank three symmetric tensor. 
\end{itemize}
\end{dfn}

This multiplication operation is  commutative, because the connection is torsionless.  
We denote this KV algebra by $(T_\mathscr{P},\circ)_{KV}$.

\subsection{Formally real Jordan algebra}\label{S:4.3}
The classification of real Jordan algebras associated to cones of Table~\ref{T:table1} is well known. This classification goes back to Jordan, von Neumann and Wigner (1934). A real Jordan algebra  $\cJ$ attached to a given cone $\mathscr{P}$ is obtained by introducing an operation $X\bullet Y =L(X)Y,$ where $L(X)$ is an endomorphism (see \cite[Theorem III.3.1]{FK} \cite[Ch.II, paragraphs 4--5]{Ko}). 
\subsubsection{}
By Vinberg's construction \cite[Ch.III]{Vin}, there exists a relation between the algebra $(T_\mathscr{P},\circ)_{KV}$ and the formally real Jordan algebra $(\cJ,\bullet)$ so that the algebra $(T_\mathscr{P},\circ)_{KV}$ generates a real Jordan algebra  $(\cJ,\bullet)$. Recall this construction.

\smallskip 

 Let $(\mathscr{U},\star)$ be a $T-$algebra of rank $m$ such that:
\begin{itemize}
\item $\mathscr{U}$ is bigraded by subspaces $\mathscr{U}_{ij}$, where $i,j= 1,\cdots, n$ such that $\mathscr{U}_{ij}\mathscr{U}_{jk}\subset \mathscr{U}_{ik}$ and $\mathscr{U}_{ij}\mathscr{U}_{lk}=0$, if $j\neq l$; 
\item the symbol $\star$ stands for an involute anti-automorphism. It is a linear mapping of $\mathscr{U}$ onto itself satisfying the following conditions: $a^{\star\star}=a;$ $(ab)^\star=b^\star a^\star $ and $\mathscr{U}_{ij}^\star\subset \mathscr{U}_{ji}.$
\end{itemize}

By \cite[Ch.III, paragraph 1]{Vin}, the algebra $\mathscr{U}$ can be decomposed into a pair of subspaces denoted $\mathscr{X}$ and $\Theta$ such that $\mathscr{U}=\mathscr{X} \oplus \Theta$,
where: \begin{itemize} \item $\mathscr{X} $ is the subspace of Hermitian matrices $\mathscr{X} :=\{ X\in \mathscr{U}\,|\,X^\star=X \}$ and 
\item $\Theta$ is the subspace of skew-Hermitian matrices $\Theta= \{ X\in \mathscr{U}\,|\,X^\star= -X \}$.\end{itemize} 
For every convex homogeneous cone $\mathscr{P}$ there exists a unique $T$-algebra $\mathscr{U}$, up to isomorphism, such that the cones $\mathscr{P}$ and $\mathscr{P}(\mathscr{U})$ are isomorphic, by \cite[Theorem 4]{Vin}.

Let $T=T(\mathscr{U})$ be a connected Lie group of upper triangular matrices in $\mathscr{U}$ with positive diagonal elements. Let $\ft$ be the associated Lie algebra to $T(\mathscr{U})$. Every Hermitian matrix can be written in the form $XX^\star $, where $X\in T(\mathscr{U})$. Therefore, the homogenous convex cone can be expressed in the form $\mathscr{P}(\mathscr{U})=\{XX^\star\, |\, X\in T(\mathscr{U})\}$ (see \cite[Ch.III, paragraph 3]{Vin}), where the Lie group $T(\mathscr{U})$ acts linearly and simply transitively on this homogeneous cone.  Let $e$ be a unit matrix of $T(\mathscr{U})$. 

Consider the map 
\begin{align*}
F: \ft \to &\quad \Xi\\
X\mapsto& \quad XX^\star.\\
\end{align*}

Then, the following isomorphism exists: 

\begin{equation}\label{E:Iso}
\begin{split}
dF: \ft \to & \quad \Xi \\
X\mapsto  & \quad X+X^\star.
\end{split}
\end{equation}

Given that the tangent space to the homogeneous cone at the origin can be identified naturally with $\mathscr{X}$ and also with the Lie algebra $\ft$ associated to the Lie group $T$ (see \cite[Ch.III, paragraph 2]{Vin}),  the isomorphism  $dF$ endows  $\ft$ with a multiplication operation $\bullet$.  This follows from the fact that the multiplication $\circ$ is defined on the tangent space to the cone $\mathscr{P}$.

Therefore, for every $a, b \in \ft$:  
\begin{equation}\label{E:bullet}
a \bullet b= dF^{-1}(dF(a) \circ dF(b)).
\end{equation}

In particular, for every $a, b, \in \ft$ we obtain
\begin{equation}\label{E:Jordan}
a \bullet b=  \frac{1}{2}(dF^{-1}(dF(a)dF(b) +  dF(b)dF(a))), 
\end{equation}
where $dF(a), dF(B)\in \mathscr{X}$ are hermitian matrices (i.e. elements $T_p\mathscr{P}$).

\subsubsection{} Given an associative algebra $(\mathscr{A},\cdot)$, a standard method provides a Jordan algebra, where the Jordan product can be given by $X\bullet Y= \frac{1}{2}(X\cdot Y+Y\cdot X)$, where $X,Y\in \mathscr{A}$ and $\cdot$ is the product operation in $\mathscr{A}$.
The tangent space at the origin $e$ to $\mathscr{P}$ is represented by the space of symmetric  matrices  (or hermitian if $\mathbb{K}$ is $\C, \hH$ or $\oO$) (see Table~\ref{tab:Jordan}). The space of symmetric matrices  (or hermitian) of a given size equipped with the standard matrix multiplication forms an associative algebra. Therefore, for every cone $\mathscr{P}$ there exists a real Jordan algebra associated to it. It is obtained by introducing the following operation 
$X\circ Y= \frac{1}{2}(XY+YX),$ where $XY$ is the classical matrix product and $X,Y$ are symmetric (or hermitian) matrices.

We invoke the following statement concerning those real Jordan algebras.

\begin{lem}\cite[Ch. VI, Theorem 12, paragraph 4, p.118]{Ko}\label{L:Jordan}
Let $\cJ$ be a real Jordan algebra associated to a cone $\mathscr{P}$ (see Table \ref{T:table1}). Then the following statements hold:

\begin{itemize}
\item $\cJ$ is formally real;
\item there exists a positive definite bilinear form $\langle-,- \rangle$ satisfying 
\begin{equation}\label{E:a}
\langle x,y\circ z \rangle=\langle x\circ y,z \rangle, \quad x,y,z \in \cJ.\end{equation}
\end{itemize}
\end{lem}
This is an important statement because it implies the existence of a positive definite bilinear form satisfying associativity on $\cJ$. 

\subsection{Pre-Frobenius cones}
The following statement provides a new proof of the fact that the cones $\mathscr{P}$ carry a pre-Frobenius structure. This allows not only to prove the existence and explicitly describe the locus of a Frobenius manifold in $\mathscr{P}$. We provide explicitly the potential function, Hessian metric and rank three symmetric tensor as well as the multiplication operation on the tangent sheaf for the pre-Frobenius structure. The proof of Theorem~\ref{T:main} does not allow such a degree of precision. 

\begin{prop}\label{P:pre-Fro}
Let  $\mathscr{P}$ be a cone such as in Table~\ref{T:table1}. Then, it carries a pre-Frobenius structure. 
\end{prop}
\begin{proof}
\begin{enumerate}[label=\alph*)]
\item Let us prove that (i) of Sec. \ref{Enum} is satisfied. Assume $\mathscr{P}$ is a cone from Table~\ref{T:table1}. Then, the cone $\mathscr{P}$ carries an affine flat structure. 
\item[]
\item[] The proof of the existence of an affine flat structure follows from works of Koszul (see for example~\cite{Kos61}). A more modern summary can be found encapsulated in~\cite[Theorem 2.22]{Bu}. This theorem states that 
there is a one-to-one correspondence between $n$-dimensional convex symmetric cones and $n$-dimensional pre-Lie algebras. In other words, $\mathscr{P}$ comes equipped with a torsionless affine flat connection $\con{0}$ (see Sec. \ref{S:KV} and Definition \ref{P:FA} for details on the construction). 

\item To prove that (ii) and (iii) of Sec.~\ref{Enum}, we need to introduce a potential function, defined everywhere locally on the cones. There exists an explicit construction of a potential function $\Phi$ defined everywhere locally on the cones. This is obtained as follows. Consider the Koszul--Vinberg (KV) function $\chi(x)$: 

\begin{equation}
\chi(x)=\int_{\mathscr{P}^*}\exp{-\langle x,a^*\rangle} da^*
\end{equation} 
where $da^*$ is a volume form invariant under translations in the dual cone $\mathscr{P}^*$ (see Appendix \ref{A:1} for details on the dual cone). The potential function $\Phi=\ln\chi(x)$ is defined everywhere on the open cone. Properties of the KV-function are enumerated as follows:
\begin{itemize}
\item The KV-function tends to infinity on the boundary of the cone. 
\item From the definition of $\chi$ we have that for any $x\in \mathscr{P}$ and any $g\in GL_n,$ one has $$\chi(gx)=|\det g|^{-1}\chi(x).$$ The differential form $\alpha=d\chi/\chi$ is invariant under $GL_n$. 
\item  We have that $\Phi=\ln\chi$ is strictly convex (by H\"older's inequality).  
\end{itemize}

 Let us prove that (ii) of Sec.~\ref{Enum} is satisfied. There exists a class of Riemannian metrics compatible with $\con{0}$ and forming a Hessian metric. The canonical Riemannian metric attached to the cone is constructed from 
$\mathrm {Hess}(\ln\chi(x))$. It is invariant under the automorphism group $GL_n.$ 

\item[]

\item[] Therefore, in local coordinates, assuming $(x^1,\cdots, x^n)$ forms an affine coordinate system on $\mathscr{P}$, the convex homogeneous domain  admits an invariant volume element defined as  $\Phi dx^1\wedge \cdots \wedge dx^n$, where $\Phi=\ln\chi(x)$.  The canonical bilinear form is: 
\begin{equation}\label{E:Riem}
g=\sum_{i,j}\frac{\partial^2\Phi}{\partial x^i\partial x^j}dx^idx^j,
\end{equation} 
where $\Phi=\ln\chi(x)$ is a real-valued convex function. The  canonical bilinear form $g$ is positively definite. This gives the Riemannian metric and defines the Hessian structure.

    \item Let us prove that (iii) of Sec.~\ref{Enum} is satisfied. There exists a symmetric rank three tensor on $\mathscr{P}.$
Because of the existence of the potential function $\Phi$ one can introduce a rank 3 symmetric tensor $A$ given by $$A_{ijk}=\partial_i\partial_j\partial_k\Phi.$$ So, we have proved that (iii) of Sec.~\ref{Enum} is satisfied.

    \item Let us prove that (iv) of Sec.~\ref{Enum} is satisfied. The multiplication operation on the tangent sheaf is precisely defined in Definition~\ref{E:circ}, in Sec. \ref{S:KV}.

    \item Let us prove that (v) of Sec.~\ref{Enum} is satisfied. In local coordinates, we have:

  \begin{align*}
g(\partial_a \circ \partial_b,\partial_c)= & g\left(\sum_{e,f}\frac{1}{2}\partial_{abf}\Phi g^{fe}p\partial_e,\partial_c\right)= \sum_{e,f}\frac{1}{2}\partial_{abf}\Phi g^{fe}g(\partial_e,\partial_c) \\
 &=\sum_{e,f}\partial_a\partial_b\partial_f\Phi g^{fe} g_{ec}=A_{abc}= \partial_a\partial_b\partial_c\Phi.\\
\end{align*}

 On the other side, 
 
 \begin{align*}
 g(\partial_a, \partial_b \circ\partial_c)=& g\left(\partial_a,\sum_{e,f}\frac{1}{2}\partial_{bcf}\Phi g^{fe}\partial_e \right)=\sum_{e,f}\frac{1}{2}\partial_{bcf}\Phi g^{fe}g(\partial_a,\partial_e)\\
 &=\sum_{e,f}\partial_b\partial_c\partial_f\Phi  g^{fe} g_{ea}=  A_{bca}=\partial_b\partial_c\partial_a\Phi.\\
   \end{align*}
    
So, by symmetry of  $A$, this gives the following  relation: $$A(X,Y,Z)=g(X\circ Y, Z)=g(X,Y\circ Z),$$ for  flat vector fields $X,Y,Z$. 
\end{enumerate}

We have thus proved that the manifold $\mathscr{P}$ is pre-Frobenius.

 \end{proof}

\subsection{Frobenius manifold in $\mathscr{P}$}\label{S:Frob}


\begin{lem}\label{L:ConnFrob}
Let $\cJ$ be a formally real Jordan algebra. Assume that there exists a subalgebra $ \mathscr{A} \subset \cJ$
which is associative, commutative and unital. Then, $\mathscr{A}$ is a Frobenius algebra. 
\end{lem}

\begin{proof}

Assume there exists a subalgebra $\mathscr{A}$ of $\cJ$ which is  commutative, associative and unital. By Lemma~\ref{L:Jordan}\, , $\cJ$  is equipped with a symmetric bilinear form satisfying the associativity condition in Eq.~\eqref{E:a}. So, it implies that $\mathscr{A}$ is a Frobenius algebra.
\end{proof}

\begin{prop}\label{L:6}
The algebra $(T_p\F,\circ)$ forms a Frobenius algebra. 
\end{prop}

\begin{proof}
\begin{enumerate}[label=\alph*)]
    \item First, let us prove that $(T_p\F, \circ)$ is an associative, commutative and unital algebra. The Cartan algebra $\tilde{\frak{a}}$  described in Proposition \ref{P:Matrix} corresponds to the tangent space $T_p\F$. The Lie algebra $\tilde{\frak{a}}$ is represented by diagonal matrices. Consider the algebra $(\tilde{\frak{a}},\bullet)$. If we use the multiplication operation such as described in Eq.~\eqref{E:Jordan}, we get that $X\bullet Y= \frac{1}{2}(XY +YX)$, where $X,Y$ are a pair of diagonal matrices. In that case, the multiplication $X\bullet Y$ coincides with the classical matrix multiplication. Therefore, it is clear that $(\tilde{\frak{a}},\bullet)$ forms an associative, commutative and unital algebra. 
    \item Results of Sec. \ref{S:4.3} and  Lemma \ref{L:ConnFrob} imply the existence of a symmetric bilinear form on this algebra, satisfying the associativity condition of Eq.~\eqref{E:a}. So, we have outlined the structure of a Frobenius algebra on $(\tilde{\frak{a}},\bullet,\langle -,-\rangle)$. 

\item By Equations~\ref{E:Iso} and~\ref{E:bullet}, there exists an isomorphism of algebras 
$(\,  \tilde{\frak{a}},\bullet \, ) \longrightarrow (\,  T_p\F, \circ\,).$ Then, $(\tilde{\frak{a}},\bullet)$ is associative if and only if $(T_p\F, \circ)$ is associative. In particular, given that $(\tilde{\frak{a}},\bullet)$ is associative, this implies that $(T_p\F, \circ)$ is associative. The algebra $(T_p\F, \circ)$ is also commutative (the connection is torsionless). Therefore, the triple $(T_p\F, \circ)$ forms an associative, commutative and unital algebra. This algebra is equipped with a symmetric bilinear form $(-,-)$ given by the Riemannian metric $g$. This bilinear form satisfies the associativity relation, as shown in part (6) of the proof of Proposition~\ref{P:pre-Fro}. So, we have proved that  $(T_p\F, \circ)$ is a Frobenius algebra, as wanted.
\end{enumerate}\end{proof}

\begin{thm}\label{T:Frobenius}
The manifold $\F$ defined in Proposition \ref{P:Matrix} is a Frobenius manifold. 
\end{thm}
\begin{proof}
Let $\F\subset \mathscr{P}$ be the totally geodesic submanifold of the cone $\mathscr{P}$ defined in Proposition~\ref{P:Matrix}. 
By Theorem~\ref{T:main} and Proposition~\ref{P:pre-Fro}, the cone $\mathscr{P}$ has a pre-Frobenius structure. Therefore, $\F$ carries also a pre-Frobenius structure. We show in two ways that $\F$ carries a Frobenius manifold structure. 
\smallskip 

\begin{enumerate}[label=\alph*)]
    \item  One way goes as follows. By Proposition \ref{L:6}, the algebra $(T_p\F,\circ)$ forms a Frobenius algebra. 
    If an algebra $\mathscr{A}$ is commutative and associative,
then the structure constants of $\mathscr{A}$ satisfy the following relations:

\begin{align*}
\text{Associativity} \quad  Eq. ~\eqref{E:Assoc}: & \quad  \Gamma_{ij}^s\Gamma_{sk}^r=\Gamma_{ik}^s\Gamma_{sj}^r,  \\ 
\text{Commutativity} \quad Eq. ~\eqref{E:comm}:& \quad  \Gamma_{ij}^s= \Gamma_{ji}^s.\\ 
\end{align*}

Using the definition of the constant structures $\Gamma_{ij}^s$ (in Definition \ref{P:FA}), the associativity relation from Eq. ~\eqref{E:Assoc} becomes:
\smallskip 

$$\sum_{l,p} \Phi_{ijl}g^{ls}\Phi_{skp}g^{pr}=
\sum_{t,w} \Phi_{ikt}g^{ts}\Phi_{sjw}g^{wr}=$$
$$=\sum_{l,p} \Phi_{jil}g^{ls}\Phi_{ksp}g^{pr}=
\sum_{t,w} \Phi_{kit}g^{ts}\Phi_{jsw}g^{wr}$$
$$=\sum_{l,p} \Phi_{jil}g^{ls}\Phi_{ksp}g^{pr}=\sum_{t,w} 
\Phi_{ikt}g^{ts}\Phi_{jsw}g^{wr}$$
$$=\sum_{l,p} \Phi_{jil}g^{ls}\Phi_{skp}g^{pr}=\sum_{t,w} 
\Phi_{ikt}g^{ts}\Phi_{sjw}g^{wr}.$$

We find that this leads to the WDVV equation. 
\smallskip 

\item The second way, is to show that $\F$ has vanishing curvature. According to \cite[Theorem 1.5]{Man99}, a pre-Frobenius manifold $(\F,g,A)$ is a Frobenius manifold if and only if the pencil of connections on $\F$ is flat.

\item[] Let $R_{X,Y}(Z)$ be the curvature tensor, where $X,Y,Z \in T_p\F$ are vector fields. The following statement holds: $ R_{X,Y}(Z)= X \circ (Y \circ Z)-Y \circ (X \circ Z)$, where $\circ$ is defined as in Eq. \eqref{E:Koszul}.
Therefore, $R_{X,Y}(Z)=0$ if and only if $X \circ (Y \circ Z)=Y \circ (X \circ Z).$

\item[] In other words,  $R_{X,Y}(Z)=0$ if and only if $X \circ (Y \circ Z)=(X \circ Y)\circ  Z$ i.e. the algebra is associative. By Proposition~\ref{L:6} the algebra is associative. Therefore, the curvature tensor vanishes on $\F$.

\item  There exists a symmetric bilinear form given by
Eq. \eqref{E:bili}. On $T_p\F$, this satisfies the associativity type of relation: $\langle X\circ Y,Z\rangle= \langle X,Y\circ Z\rangle,$ where $X,Y,Z\in T_p\F$. This concludes the proof. 
\end{enumerate}
\end{proof}

\subsection{Some remarks}
A Lie algebra $(\mathscr{A},[-,-])$ is said to be {\it Lie associative} if $[[Y,Z],X]=[Y,[Z,X]]$, for given $X,Y,Z\in \mathscr{A}$.
\begin{lem}
Assume $\F=\exp{\tilde{\frak{a}}}$ is a flat totally geodesic submanifold in a symmetric space $\mathscr{P}$, as defined in Proposition \ref{P:Matrix}. Then, the Lie algebra $(\tilde{\frak{a}},[-,-])$ is Lie associative.
\end{lem}
\begin{proof}
   By  Theorem \ref{T:Hel}, the curvature tensor $R_0$ evaluated at $T_p\mathscr{P}$ is given by 
\begin{equation}\label{E:courb}
    R_0(X,Y)Z=-[[X,Y],Z], \quad X,Y, Z \in T_p\mathscr{P}.
\end{equation}
By Theorem \ref{T:Frobenius}, we can assume that the totally geodesic submanifold $\F$ is flat. So, restricting attention to $\F$, the left hand side of Eq. \eqref{E:courb} is 0. Thus, we have
\begin{equation}\label{E:0}
0=-[[X,Y],Z],
\end{equation} where $X,Y,Z \in T_p\F$. Now, by Jacobi's identity: 
\[[X,[Y,Z]]+[Y,[Z,X]]+[Z,[X,Y]]=0.\] Rewriting it, one gets $[[Y,Z],X]+[[Z,X],Y]=-[[X,Y],Z]$.
So, inputting this relation in Eq. \eqref{E:0} one gets
$[[Y,Z],X]=[Y,[Z,X]]$ for $X,Y,Z \in T_p\F$. Thus, $(\tilde{\frak{a}},[-,-])$ satisfies the Lie associativity.  \end{proof}

\subsection{Final remark}
The cones discussed in Table \ref{T:table1}  are also examples of Cartan--Hadamard spaces (see \cite{Shi84}, for an introduction). We conclude as follows: 
\begin{cor}\label{T:Final}
Let $\mathscr{P}$ be a non-compact symmetric space from Table~\ref{T:table1}. 
There exists an isometric immersion of 
 a flat torus being a totally geodesic submanifold of $\mathscr{P}$. It carries the structure of a Frobenius manifold; all geodesics of $\mathscr{P}$ lie in that subspace. The homogeneous Cartan--Hadamard space has  negative sectional curvature and carries a pre-Frobenius structure. 
\end{cor}
    \begin{proof}
By Nomizu~\cite{No54}, every irreducible symmetric space $G/K$ belongs to one the three main classes: euclidean, compact or non-compact. 
\begin{itemize}
        \item  It is euclidean if the curvature is 0. It is therefore isometric to a Euclidean space.
  
    \item   It is compact if the sectional curvature is non-negative (but not identically zero).

      \item  It is non-compact if the sectional curvature is non-positive (but not identically zero).       
   \end{itemize}
The symmetric spaces considered in Table~\ref{T:table1} have non-positive scalar curvature.
We decompose those symmetric spaces into a union of subspaces, where one subspace has flat sectional curvature and the other one has negative sectional curvature. It turns out from Proposition \ref{P:exist} that there exists a totally geodesic submanifold with vanishing curvature. We can refine our statement, by using Proposition \ref{P:Matrix}: those subspaces form a flat torus and are Frobenius manifolds Theorem \ref{T:Frobenius}.

The space $\mathscr{P}$ 
is a homogeneous cone, with negative sectional curvature. This is a homogeneous Cartan--Hadamard cone. It carries the pre-Frobenius structure (see Proposition~\ref{P:pre-Fro}).      
\end{proof}

\section{Conclusion}
 
$\diamond$ By Theorem \ref{T:LG} the LG theory furnishes a Hilbert space which is parametrised by a real Monge--Ampère domain which is  a potential pre-Frobenius domain.  This Hilbert space has the structure of a weighted projective space. Therefore, using the result of \cite{CR} we can show that one can build two mirror dual Calabi--Yau  hypersurfaces which are parametrised by this Monge--Ampère domain (in Theorem \ref{T:FINAL}). Therefore, our approach allows to create a bridge between the LG theory and the mirror symmetry \`a la Kontsevich--Soibelman.

$\diamond$ Our method can be considered as a possible tool for finding new sources of Frobenius manifolds, since we have shown that the Elliptic Monge--Ampère equation and the Associativity Equations relate. By Theorem \ref{T:main} the Elliptic Monge--Ampère manifold is a pre-Frobenius manifold.

$\diamond$ Our study of the LG theory implies considering a toy model formed from cones of positive definite symmetric matrices with coefficients in a division algebra $\mathbb{K}$. These cones turned out to form a Monge--Ampère domain and to form thus a pre-Frobenius domain. Given that these cones parametrise manifolds of probability distributions obeying to Wishart laws (see ~\cite{AnWo,Wi28} and \cite{Co63,Co66,Ja,CoMu72,CoMu76}), this strengthens the result \cite{CoMa} in information geometry. 

$\diamond$ If the cone is real it provides a nice example of a Monge--Ampère manifold parametrizing Calabi--Yau manifolds (torus/abelian variety), illustrating thus  \cite{KoS01} and torus fibrations. 

$\diamond$ Other applications in number theory of this example follow from the works of Minkowski, Siegel~\cite{Sie35,Sie44}, Maass~\cite{Maa}, Piateski--Shapiro~\cite{PS} (and many others).

$\diamond$ Applying the main statement of \cite{Connes} and our results, it follows that if the cone is complex one deduces a direct relation from pre-Frobenius manifolds and Monge--Ampère domains to von Neumann algebras.

\appendix 
\section{Convex symmetric cones}\label{A:1}

\subsubsection{Notations}
Let $\mathscr{P} \subset V$ be a convex cone in an $n$-dimensional vector space $V$, over the real number field. 
$\K$ is a division algebra; $\mathscr{P}_n(\K)$ is the irreducible symmetric cone of $n\times n$  positive definite matrices with coefficients in $\K$. 

\subsubsection{Strictly convex cones} 
In the following we always consider {\it strictly convex cones}. Note that for brevity we simply refer to them as {\it convex cones.} 
Recall some elementary notions (see \cite{FK} for further information). 

\begin{dfn}
Let $V$ be a finite dimensional real vector space. Let $\langle-,-\rangle$ be a non-singular symmetric bilinear form on $V$.
A subset $\mathscr{P} \subset V$ is a convex cone if and only if $x,y \in \mathscr{P}$ and $\lambda,\mu >0$ imply $\lambda x+\mu y \in \mathscr{P}$.
\end{dfn} 

\subsubsection{Homogeneous cones} 
The automorphism group $G(\mathscr{P})$ of an open convex cone $\mathscr{P}$ is defined by 
\[G(\mathscr{P})=\{g\in GL(V)\, |\, g\mathscr{P}=\mathscr{P}\}\]
An element $g\in GL(V)$ belongs to $G(\mathscr{P})$ iff $g\overline{\mathscr{P}}=\overline{\mathscr{P}}$ \cite{FK}. So, $G(\mathscr{P})$ is a closed subgroup of $GL(V)$ and forms a Lie group. 
The cone $\mathscr{P}$ is said to be {\it homogeneous} if $G(\mathscr{P})$ acts transitively upon $\mathscr{P}$.

\smallskip 

\subsubsection{Symmetric cones} 
From homogeneous cones one can construct symmetric convex cones. Let us introduce the definition of an open dual/polar cone. An open dual/polar cone $\mathscr{P}^*$ of an open convex cone is defined by $\mathscr{P}^*=\{y\in V\, |\, \langle x,y \rangle>0,\, \forall\, x\in \overline{\mathscr{P}}\setminus 0 \}$. A homogeneous convex cone $\mathscr{P}$ is symmetric if $\mathscr{P}$ is self-dual i.e. $\mathscr{P}^*=\mathscr{P}$. Note that if $\mathscr{P}$ is homogeneous then so is $\mathscr{P}^*$. 
 
 \begin{rem} 
 If $\mathscr{P}$ is a symmetric open cone in $V$,  then $\mathscr{P}$ is a symmetric Riemann space.    
\end{rem}

\smallskip

\subsubsection{Automorphism group of symmetric cones}  Let us go back to the automorphism group of $\mathscr{P}$. This discussion relies on \cite[Prop I.1.8, Proposition I.1.9]{FK}.

\smallskip 

Let $\mathscr{P}$ be a symmetric cone in $V$.  For any point $a\in \mathscr{P}$ the stabilizer of $a$ in $G(\mathscr{P})$ is given by 
\[G_a=\{g\in G(\mathscr{P})\, |\, ga=a\}.\]

By \cite[Prop I.1.8]{FK}, if $\mathscr{P}$ is a proper open homogeneous convex cone then for any $a$ in $\mathscr{P}$, $G_a$ is compact. Now, if $H$ is a compact subgroup of $G$ then $H\subset G_a$ for some $a$ in $\mathscr{P}$. This means that the groups $G_a$ are all maximal compact subgroups of $G$ and that if $\mathscr{P}$ is homogeneous then all these subgroups are isomorphic. 

By \cite[Proposition I.1.9]{FK}, if $\mathscr{P}$ is a symmetric cone, there exist points $e$ in $\mathscr{P}$ such that $G(\mathscr{P})\cap O(V)\subset G_e$, where $O(V)$ is the orthogonal group of $V$. For every such $e$ one has $G_e=G\cap O(V)$ 

Suppose $\mathscr{P}$ is a convex homogeneous domain in $V$. Assume that
\begin{itemize}
   \item[---]   $G(\mathscr{P})$ is the group of all automorphisms;
   \item[---]   $G_e=K(\mathscr{P})$ is the stability subgroup for some point $x_0\in \mathscr{P}$;
   \item[---]   $T(\mathscr{P})$ is a maximal connected triangular subgroup of $G(\mathscr{P}).$ 
\end{itemize}

\smallskip 

Following \cite[Theorem 1]{Vin}, we have: \[G(\mathscr{P})=K(\mathscr{P})\cdot T(\mathscr{P}),\] where $K(\mathscr{P}) \cap T(\mathscr{P}) = e$ and the group  $T$ acts simply transitively. 

This decomposition on the Lie group side leads naturally to its Lie algebra. Cartan's decomposition for the Lie algebra tells us that $\fg=\textgoth{k} \oplus\ft,$ 

where:

\begin{itemize}
   \item[---]   $\frak{t}$ can be identified with the tangent space of $\mathscr{P}$ at $e$. 

   \item[---]   $\textgoth{k}$ is the Lie algebra associated to $K(\mathscr{P})$
\end{itemize}
 and
\[[\ft,\ft]\subset \textgoth{k},\quad [\textgoth{k},\ft]\subset \ft.\]

\begin{thm}\cite[Theorem 3.3]{Hel}

\begin{enumerate}
\item Let $M$ be a Riemannian globally symmetric space and $p$ is any point in $M$. If $G$ is a Lie transformation group of $M$ (a Lie group) and $K$ is the subgroup of $G$ which leaves $p$ fixed, then $K$ is a compact subgroup of $G$ 
and $G/K$ is analytic diffeomorphic to $M$ under the mapping $gK\to g\cdot p$, $g\in G$. 
\item The mapping $\sigma$ is an involutive automorphism of $G$ such that $K$ lies between the closed subgroup $K_\sigma$ of all fixed points of $\sigma$ and the identity component of $K_{\sigma}$. The group $K$ contains no normal subgroup of $G$ other than $\{e\}$.
\item Let $\fg$ and $\textgoth{k}$ denote the Lie algebras of $G$ and $K$, respectively. Then $\textgoth{k} =\{ X\in \fg: (d\sigma)_eX=X\} $ and if $\ft=\{X\in \fg\, |\, (d\sigma)_eX=-X\}$ we have $\fg=\textgoth{k} \oplus \ft$. Le $\pi$ be the natural mapping $g\to g\cdot p$ of $G$ onto $M$.
Then $(d\pi)_e$ maps $\textgoth{k}$ into $\{0\}$ and $\ft$ isomorphically onto $T_pM$. If $X\in \ft$ then the geodesic emanating from $p$ with tangent vector $(d\pi)_eX$ is given by $\exp{tX}\cdot p$. Moreover, if $Y\in T_pM$, then $(d\exp{tX})_p(Y)$ is the parallel translate of Y along the geodesic. 
\end{enumerate}
\end{thm}

\subsubsection{Classification}
Any symmetric cone (i.e. homogeneous and self-dual) $\mathscr{P}$ is in a unique way isomorphic to the direct product of irreducible symmetric cones $\mathscr{P}_i$ (cf. \cite[Proposition III.4.5]{FK}). 
\medskip 

\begin{prop}~\label{P:Vclass}
Each irreducible homogeneous self--dual cone belongs to one
of the following classes:

\smallskip

\begin{table}[ht]
    \centering
    \begin{tabular}{|c|c|l|}
  \hline
Nb & Symbol & Irreducible symmetric cones \\
 \hline
 1. &      $ \mathscr{P}_n(\R)$ &  Cone of $n \times n$ positive definite symmetric real matrices. \\
     &  & \\
      
    2.  &      $ \mathscr{P}_n(\C)$ &  Cone of $n \times n$ positive definite self-adjoint complex matrices. \\
&  & \\
       3.  &    $ \mathscr{P}_n(\hH)$ &  Cone of $n \times n$ positive definite self-adjoint quaternionic matrices. \\
           &   & \\
        4. &   $ \mathscr{P}_3(\oO)$ & Cone of $3 \times 3$  positive definite self-adjoint octavic matrices. \\
           &   & \\
    5. &    $\Lambda_n$    & Lorentz cone  given by $x_0>\sqrt{\sum_{i=1}^n x_i^2}$ (spherical cone). \\
        &   & \\
       \hline
    \end{tabular}
    \caption{Irreducible symmetric cones}
    \label{tab:cones}
\end{table}
\end{prop}

 \subsubsection{Symmetric cones}\label{S:3.1}
The cones of positive definite quadratic forms are non-compact symmetric spaces. A {\it symmetric space} is a Riemannian space, which can be written as the quotient of Lie groups $G/K$, where $G$ is a connected Lie group with an involutive automorphism whose fixed point set is essentially the compact subgroup $K\subset G.$ 

The pair $(G,K)$ is a {\it symmetric pair} provided that there exists an involution $s\in G$, such that  $(K_{s})_{0} \subset K \subset K_{s}$, where $K_{s}$ is the set of fixed points of $s$ and $(K_{s})_{0}$ is the identity component of $K_{s}$. See  \cite[Chp.IV, paragraphs 3, 4 ]{Hel} for a detailed exposition.
\subsubsection{Classification Table}
One identifies the space of  $n\times n$ symmetric (resp. hermitian) positive definite matrices over a real division algebra with the following non-compact symmetric spaces $GL_n(\R)/O_n$ (resp. $GL_n(\C)/U_n$, $GL_n(\hH)/Sp_n$ $GL_3 (\oO)/F_4$) (see \cite[p.97]{FK}). This is summarised in the following classification table. 
\begin{table}[ht]
    \centering
\begin{tabular}{|c|c|c|c|c|c|}
     \hline
Cone  & $G/K$ & $T_p\mathscr{P}_n$  & $\frak{g}$ & $\frak{k}$  \\
 \hline
 $\mathscr{P}_n(\R)$ & $GL_n(\R)/O_n $& $Sym(n,\R)$ & $ \frak{sl}(n,\R)\oplus \R$ & $\frak{o}$(n) \\
 $\mathscr{P}_n(\C)$ &  $GL_n(\C)/U_n$ &$Herm(n,\C)$  &  $ \frak{sl}(n,\C)\oplus \R$ & $\frak{su}$(n)  \\
 $\mathscr{P}_n(\hH)$ & $GL_n(\hH)/Sp_n$ & $Herm(n,\hH)$ &  $\frak{sl}(m,\hH)\oplus \R$ & $\frak{su}(n,\hH$) \\
 $\mathscr{P}_3(\oO)$ & $GL_3 (\oO)/F_4$  &$ Herm(3,\oO)$ & $ \frak{e}_{(-26)}\oplus \R$ & $\frak{f}_4$ \\
  \hline
\end{tabular}
\caption{}\label{T:table1}
\end{table}

To clarify the notations:
\begin{itemize}
 \item $T_p\mathscr{P}_n$ is the tangent space to the cone at a point $p$ in $\mathscr{P}_n$.
\item  $Sym(n,\mathbb{K})$ stands for the space of $n\times n$ symmetric matrices defined over $\mathbb{K}$;
\item  $Herm(n,\mathbb{K})$ denotes the space of $n\times n$ self-adjoint matrices  defined over $\mathbb{K}$; 
 \item $\frak{g}$ is the Lie algebra associated to $G$;
 \item $\frak{k}$  is the Lie algebra associated to $K$;
\end{itemize}

The tangent space to $\mathscr{P}$ at a point carries a Jordan algebra structure. We recall this: 
\begin{table}[ht]
    \centering
    \begin{tabular}{|c|c|c|c|}
     \hline
    &&& \\
 Irreducible  & Formally real simple  & $dim \cJ $ & rank $\cJ $  \\
 symmetric cone & Jordan algebras $\cJ$ & & \\
& && \\
 \hline
& && \\
$ \mathscr{P}_n(\R)$  &  Jordan algebra of $n\times n$ self-adjoint real matrices & $\frac{1}{2}n(n+1)$ & $ n$  \\

& && \\
    $ \mathscr{P}_n(\C)$      &  Jordan algebra of $n\times n$  self-adjoint complex matrices & $n^2$ & $ n$    \\
   & && \\
    
    $ \mathscr{P}_n(\hH)$     &   Jordan algebra of $n\times n$  self-adjoint quaternionic matrices & $n(2n-1)$ & $n$   \\
  & && \\
    
     $ \mathscr{P}_3(\oO)$    &  Jordan algebra of $3\times 3$ self-adjoint octonionic matrices: &   27 & 3   \\
       &  Albert algebra. &  &\\
      & && \\
      \hline
    \end{tabular}
    \caption{Simple formally real Jordan algebras}
    \label{tab:Jordan}
\end{table}

\subsubsection{Symmetric bilinear forms}
It is a well known fact that if $G$ is semi-simple, the Killing bilinear form is non-degenerate on $\fg$. The symmetric bilinear form is given by:  
$$\langle X,Y\rangle=-Tr(ad X\, ad Y),$$ where $ad\, Z(\xi)=[Z,\xi]$ and $Z \in \fg$. Therefore, we may state the following:

\begin{prop}
Let  $\mathscr{P}$ be a cone listed in Table \ref{T:table1}. Then, this irreducible cone comes equipped with a $G$-invariant metric and with a symmetric bilinear
form given by 
\begin{equation}\label{E:bili}
    \langle X,Y\rangle=\Re\, Tr(XY),
\end{equation} where $X,Y\in T_p(Gl_n(\K)/K)\cong \ft\subset \fg$,  where $Tr(\cdot)$ stands for the trace operator and $\Re$ is the real part.
\end{prop}
\begin{proof}
This statement follows from the existence of the Killing form. See \cite[p. 46, Proposition III.1.5]{FK} for a precise statement. 
\end{proof}

\section{Landau--Ginzburg model}\label{S:LG}
We compare this construction with the LG model. Mathematically the LG model is summarised as a non-compact K\"ahler manifold and a holomorphic Morse function.  To improve the mathematical understanding, we propose to recall additional information which are present in the original construction. 

\subsubsection{}
 
In 1957, Bardeen, Cooper, and Schrieffer (BCS) introduced the foundation for a quantum theory of superconductivity. This gave the BCS theory.
\begin{itemize}
    \item There exists an important object:  the (BCS) {\it wave function} $\psi$. It is a function of two-particle system $k \uparrow$ and -$k \downarrow$ ($\uparrow$ and $\downarrow$ stand for the spin). 
    \item In the BCS state, one-particle orbitals are occupied in pairs: if an orbital with wave vector $k$ and spin up is occupied, the orbital with wave vector $k$ and spin down is also occupied. If $k \uparrow$ is vacant, then so is $ k \downarrow$ vacant.
    \end{itemize}
The pairs are called {\it Cooper pairs} and they have spin zero. 

For a complete set of states of a two-electron system satisfying periodic boundary conditions in a cube of unit volume, take the plane wave product functions $\varphi(k_1,k_2;r_1,r_2)=\exp(\imath(k_1\cdot r_1+k_2\cdot r_2))$. It is assumed that the electrons are of opposite spin.
So, the plane wave can be expressed as $\varphi(K,k;R,r)=\exp(\imath K\cdot R)\exp(\imath k\cdot r),$ where we have:

 the barycentre $R$ of $r_1$ and $r_2$, $R=\frac{1}{2}(r_1+r_2)$ and we write $r=r_1-r_2$; Reciprocally,
$K=k_1+k_2$ and $k=\frac{1}{2}(k_1-k_2)$.

\subsubsection{}
Let us introduce the order parameter $\psi(r)$ such that  $\overline{\psi(r)}\psi(r)=|\psi(r)|^2=n(r)$. This is a complex valued function. It corresponds to the local concentration of superconducting electrons. It represents a condensed wave function which is a single quantum state occupied by a large fraction of Cooper pairs. In other words, this generates the {\it density of probability} of finding Cooper pairs in a given domain.

\subsubsection{}
 We explain how the holomorphic Morse function mentioned above is obtained. It corresponds to the free energy density, expressed as a function of $\psi(r)$ as follows: 

\begin{equation}
F_s(r)=F_0-\alpha |\psi|
^2 +\frac{\beta}{2}|\psi|^4 + \frac{1}{2m}|-\imath \hbar  \nabla- q\frac{A}{c}\psi |^2 - \int_0^{B_a} M \cdot dB_a
\end{equation}
where:
\begin{itemize}
\item $\alpha, \beta,m$ are positive constants;
\item $F_0$ is the free energy density of the normal state;
\item $- \alpha |\psi |
^2 +\frac{\beta}{2}|\psi|^4$ is a Landau form for the expansion of free energy vanishing at a second-order phase transition;
\item the term in $| \nabla |\psi |^2$ represents an increase in energy caused by a spatial variation of the order parameter. 
\item the term  $ \int_0^{B_a} M \cdot dB_a$ represents the increase in the superconducting free energy.
\end{itemize}
\subsubsection{Landau--Ginzburg equation}
In order to obtain the Landau--Ginzburg equation it is necessary to minimise the total free energy $\int F_s(r)dV$ w.r.t $\psi(r)$.
We get 
\[\delta F_s(r)=[-\alpha \psi + \beta |\psi |^2 \psi +\frac{1}{2m}(-\imath \hbar \nabla-q\frac{A}{c})\delta\overline{\psi}+c.c.]\]
If $\delta \overline{\psi}$ vanishes on the boundaries, it follows that 
$\int \delta F_s(r)dV=\int dV \delta\overline{\psi} [-\alpha \psi + \beta |\psi |^2\psi +\frac{1}{2m}(-\imath \hbar \nabla-q\frac{A}{c})]+c.c$.
The integral vanishes if the term in the bracket is null. That is if
\begin{equation}
    [\frac{1}{2m}(-\imath \hbar \nabla-q\frac{A}{c})^2 -\alpha \psi + \beta | \psi |^2]\psi=0
\end{equation} 
This last equation is precisely the LG equation. 

By minimising the free energy $F_s$ wrt $\delta A$ one gets a gauge-invariant expression for the super current flux:
\[j_s(r)=\frac{\imath g\hbar}{2m}[\overline{\psi},\nabla\psi] -\frac{q^2}{mc}|\psi|^2A.\]

\vfill\eject

\end{document}